\documentclass[11pt]{amsart}
\usepackage{latexsym, amsfonts, amsmath, amssymb, amsthm, cite, tabls, paralist, xcolor, comment, enumerate, enumitem}
\usepackage[pdftex,citecolor=green,linkcolor=red]{hyperref}

\newtheorem{theorem}{Theorem}[section]
\newtheorem{lemma}[theorem]{Lemma}
\newtheorem{proposition}[theorem]{Proposition}

\newtheorem{corollary}[theorem]{Corollary}

\newcommand{\nor}[1]{\lVert #1 \rVert}
\newcommand{\abs}[1]{\left\lvert #1 \right\rvert}

\theoremstyle{remark}

\theoremstyle{definition}
\newtheorem{definition}[theorem]{Definition}

\theoremstyle{remark}

\numberwithin{equation}{section}

\newcommand{\PMod}[1]{\,(\textnormal{mod}\,#1)}
\newcommand{\Z}{\mathbb{Z}}

\def\Vol{\mathrm{Vol}}
\newcommand{\R}{\mathbb{R}}
\newcommand{\F}{\mathbb{F}}
\newcommand{\N}{\mathbb{N}}
\newcommand{\C}{\mathbb{C}}
\newcommand\p{{\mathcal{P}}}
\newcommand\D{{\mathcal{D}}}
\newcommand\cH{{\mathcal{H}}}
\newcommand{\E}{\mathbb{E}}
\newcommand{\ve}{\mathbf}

\newcommand{\ee}{\varepsilon}
\newcommand{\CB}{\mathcal{B}}

\newcommand{\CL}{\mathcal{L}}

\newcommand{\CX}{\mathcal{X}}
\newcommand{\FS}{\mathfrak{S}}

\newcommand{\Prime}{\mathbb{P}}

\usepackage{graphicx}
\usepackage{mathrsfs}
\usepackage{varioref}
\usepackage{enumerate}

\begin{document}

\title[A transference principle for linear equations]{A transference principle for systems of linear equations, and applications to almost twin primes}

\author{Pierre-Yves Bienvenu}
\address{Institut f{\"u}r Analysis und Zahlentheorie\\TU Graz\\Kopernikusgasse 24\\8010 Graz, Austria}
\email{bienvenu@tugraz.at}

\author{Xuancheng Shao}
\address{Department of Mathematics, University of Kentucky, Lexington, KY, 40506, USA}
\email{xuancheng.shao@uky.edu}

\author{Joni Ter\"av\"ainen}
\address{Mathematical Institute\\ Radcliffe Observatory Quarter\\ Woodstock Road\\ Oxford OX2 6GG \\ United Kingdom}
\email{joni.teravainen@maths.ox.ac.uk}

\begin{abstract}
The transference principle of Green and Tao enabled various authors to transfer Szemer\'edi's theorem on
long arithmetic progressions in dense sets to various sparse sets of integers, mostly sparse sets of primes. In this paper, 
we provide a transference principle which applies to general affine-linear configurations of finite complexity. 

We illustrate the broad applicability of our transference principle with the case of \emph{almost twin primes}, by which we mean either Chen primes or
``bounded gap primes'', as well as with the case of primes of the form $x^2+y^2+1$. Thus,
we show that in these sets of primes the existence of solutions to finite complexity systems of linear equations is determined
by natural local conditions. These applications rely on a recent work of the last two authors
on Bombieri--Vinogradov type estimates for nilsequences.
\end{abstract}

\maketitle

\section{Introduction}
\subsection{The problem and its background}
Green and Tao famously proved \cite{GT-kAPinPrime} that the primes contain arbitrarily
long arithmetic progressions. Their proof introduced an influential transference principle, stating that if a set of integers is dense inside a  \emph{pseudorandom} set, then it contains arbitrarily long arithmetic progressions. This is called a transference principle, since it transfers Szemer\'edi's theorem, which states that any dense subset of $\mathbb{Z}$ contains arbitrarily long arithmetic progressions, to a sparse setting. In fact, the proof of Green and Tao relied on Szemer\'edi's theorem as a black box.

More generally, given any admissible\footnote{\label{foot0}We say that $\Psi=(\psi_1,\ldots,\psi_t)$ is \textit{admissible} if $(\psi_i(\mathbf{n}))_{\mathbf{n}\in \mathbb{Z}^d}$ has no fixed prime divisor for each $i\in[t]$.} affine-linear map $\Psi : \Z^d\rightarrow\Z^t$, and a subset $\mathcal{P}$ of the primes, one may ask whether $\mathcal{P}^t$ contains a tuple of the form $\Psi(\mathbf{n})$ with $\mathbf{n}\in\Z^d$. Since the image of an affine-linear map may always be realised as the kernel of another affine-linear map and vice-versa, this may be formulated as the problem of determining which linear systems of equations can be solved inside $\mathcal{P}$. 

Since the Green--Tao theorem, a lot of research has been devoted to this question. Note that $k$-term arithmetic progressions correspond to the map $\Psi(n,d)=(n,n+d,\ldots,n+(k-1)d)$, so this case is handled by the Green--Tao theorem, which actually holds for dense subsets of the primes (or even not too sparse subsets, see \cite{Rimanic}). Further, since Szemer\'edi's theorem holds for any given translation-invariant linear configuration in place of arithmetic progressions (that is, for homogeneous linear maps $\Psi$
 such that $(1,\ldots,1)\in\Psi(\Z^d)$),
the Green--Tao theorem also holds for these linear configurations.

Regarding general linear configurations, under the mere assumption that $\Psi=(\psi_1,\ldots,\psi_t)$ has finite complexity, (that is, no two of the forms $\psi_i$ are affinely related), Green and Tao \cite{GT-linear-equations} provided
a complete answer in the case where $\mathcal{P}$ is the full set $\mathbb{P}$ of primes, in fact giving an asymptotic formula for the number of
$\mathbf{n}\in [N]^d$ for which $\Psi(\mathbf{n})\in \mathcal{P}^t$ as $N\rightarrow \infty$.

Regarding subsets of the primes, it is known that a number of interesting sparse subsets of the primes contain arbitrarily long arithmetic progressions (or again, any given translation-invariant linear configuration). Indeed, the \textit{Chen primes} 
\begin{align*}
\mathcal{P}_{\textnormal{Chen}}:=\{p\in \mathbb{P}:\,\, p+2\in P_2\},    
\end{align*}
where $P_2$ is the set of integers which have at most two prime factors (counted with multiplicity),
have this property by \cite{kAP-Chen}, and the \emph{bounded gap primes}
\begin{align*}
\mathcal{P}_{\textnormal{bdd},H}:= \{n\in \mathbb{P}:\,\, |[n,n+H]\cap \mathbb{P}|\geq 2\}  \end{align*}
for large $H$  have this property by  \cite{kAP-Maynard,kAP-Maynard2}.  Primes of the form $x^2+y^2+1$ by \cite{kAP-twoSquaresPlusOne} have this property as well and for any $k$, there exists $c_k>1$ such that
for any $c\in [1,c_k)$ the set $\mathbb{P}\cap\{\lfloor n^c\rfloor : n\in\N\}$ of Piatetski--Shapiro primes contains progressions of length $k$ by \cite{li2019greentao}

However, very little is known when $\Psi$ is not translation-invariant and simultaneously $\mathcal{P}$ is not the full set of
the primes. When $\mathcal{P}\subset \mathbb{P}$ is the (dense) set of the shifted squarefree primes (i.e. primes $p$ such that $p-1$ is squarefree), for any finite complexity $\Psi$, an asymptotic for the number of
$\mathbf{n}\in [N]^d$ for which $\Psi(\mathbf{n})\in \mathcal{P}$ was proven by the first author \cite{bienvenu}.
When it comes to non-translation-invariant configurations in subsets of the primes, previous research has concentrated on the ternary Goldbach system, that is, 
the affine-linear map
$\Psi_N(n,m)=(n,m,N-n-m)$.
The subsets of the primes where it was studied include 
 subsets of relative density above a certain threshold \cite{Li-Pan,shaoDensity}, the set of primes of the form $x^2+y^2+1$
 \cite{joni}, the set of primes in a given Chebotarev class \cite{kane},
 the set of Fouvry--Iwaniec primes  $x^2+p^2$ with $p$ prime \cite{grimmelt},
and  the set of primes admitting a given primitive root \cite{FKS}.
More relevantly for the present study, Matom\"aki and the second author \cite{matomaki-shao} showed that any sufficiently large odd integer (resp. integer congruent to three modulo six) is a sum of three bounded gap primes (resp. three Chen primes). Both of these types of primes have properties akin to those of twin primes, and are therefore referred to as \emph{almost twin primes}.

We mention that all of the papers \cite{matomaki-shao,joni,grimmelt,kane,FKS,shaoDensity} rely on classical Fourier analysis, which is considerably simpler than higher order Fourier analysis, and hence the proofs do not adapt to any systems $\Psi$ of complexity at least $2$. The papers \cite{kAP-Chen,kAP-Maynard,kAP-Maynard2,kAP-twoSquaresPlusOne,li2019greentao}, in turn, all use the Green--Tao transference principle, and hence the proofs do not adapt to any non-translation-invariant configurations $\Psi$. Our main result handles the case of arbitrary finite complexity systems $\Psi$ when $\mathcal{P}$ is the set of almost twin primes. 


\subsection{Results on linear equations}

Now we state our results on linear equations in almost twin primes precisely.  
Let $\mathcal{H} = \{h_1,\ldots,h_m\}$ be an admissible $m$-tuple: for every prime $p$, there exists $n\in\N$ such that
$\prod_{i\in[m]}(n+h_i)\not\equiv 0\PMod p$.
Let $\mathcal{P}_{\mathcal{H}}:=\{n\in \mathbb{N} : \abs{(n+\mathcal{H})\cap\mathbb{P}}\geq 2\}$.
Note that $\mathcal{P}_{\mathcal{H}}$ is actually not a subset of the primes; in fact $\mathcal{P}_{\textnormal{bdd},H}=\mathbb{P}\cap \bigcup_{\mathcal{H}\subset [0,H]}\mathcal{P}_{\mathcal{H}}$. Define the weighted indicator functions of the Chen primes and $\mathcal{P}_{\mathcal{H}}$ by
\begin{align*}
\theta_{1}(n):=(\log n)^{2}1_{\mathcal{P}_{\textnormal{Chen}}}(n)1_{p\mid n(n+2)\implies p\geq n^{1/10}},\quad \theta_2(n):=(\log n)^{m}1_{\mathcal{P}_{\mathcal{H}}}(n)1_{p\mid \prod_{i=1}^{m}(n+h_i)\implies p\geq n^{\rho}},    
\end{align*}
where $m=\abs{\mathcal{H}}\geq 2$ and $\rho\in (0,1)$.

Then we know by Chen's theorem \cite{chen} and  Maynard's theorem \cite{MaynardII}, upon assuming that $m$ is large enough and $\rho$ is small enough, that $\sum_{n\leq N}\theta_i(n)\gg N$ for $i\in\{1,2\}$ (and we have upper bounds of the same order of magnitude by Selberg's sieve). Throughout this paper, we fix such $m,\rho$, and we also fix an admissible $m$-tuple $\mathcal{H}$ in the definition of $\theta_2$.

\begin{theorem}[Arbitrary linear configurations weighted by almost twin primes]
\label{th:linEqTwinPrimes}
Let $i\in\{1,2\}$.
Let $\eta>0$, $N,d,t,L\geq 1$, and let $\Psi=(\psi_1,\ldots,\psi_t):\Z^d\rightarrow\Z^t$ be a system of affine-linear forms of finite complexity whose homogeneous parts have coefficients bounded in modulus by $L$.

Then there exists a constant $C_i(\Psi)\geq 0$ such that the following holds. 
Let $K\subset [-N,N]^d$ be a convex body satisfying
$\mathrm{Vol}(K)\geq \eta N^d$ and
 $\Psi(K)\subset [1,N]^t$.
 Then, for $N\geq N_0(L,\eta,d,t)$, we have
 \begin{equation}
 \label{eq:linEqTwinPrimes}
 \sum_{\mathbf{n}\in K\cap \Z^d}\prod_{j\in [t]}\theta_i(\psi_j(\mathbf{n}))\gg_{L,\eta,d,t} C_i(\Psi)\Vol(K).
 \end{equation} 
 Further $C_i(\Psi)>0$ unless there is an obstruction modulo some prime $p$.
 More precisely, $C_1(\Psi)>0$ as soon as for every prime $p$ there exists $\mathbf{n}\in\Z^d$
 such that $\prod_{i\in[t]}\psi_i(\mathbf{n})(\psi_i(\mathbf{n})+2)\not\equiv 0\PMod p$
 and $C_2(\Psi)>0$ as soon as for every prime $p$ there exists $\mathbf{n}\in\Z^d$
 such that $\prod_{i\in[t]}\prod_{j\in [m]}(\psi_i(\mathbf{n})+h_j)\not\equiv 0\PMod p$.
 \end{theorem}
 Note that the hypotheses imply that the non-homogeneous coefficients of $\Psi$ are bounded in modulus by $(dL+1)N$.
 It turns out that $C_i(\Psi)\gg_{d,t,L,i}1$ whenever $C_i(\Psi)>0$; therefore the right-hand side of
 the estimate \eqref{eq:linEqTwinPrimes} is $\gg_{d,t,L,i} \Vol(K)$ whenever it is not 0.
 
 We can also obtain an analogous result for primes of the form $x^2+y^2+1$. Let
\begin{align}\label{equ21}
 \theta_3(n):= (\log(2n))^{3/2}1_{\mathbb{P}}(n)1_{n=x^2+y^2+1, x,y\in \Z}  
\end{align}
be the weighted indicator function of primes of the form $x^2+y^2+1$. By a result of Iwaniec \cite{iwaniec-quadraticform}, we have $\sum_{n\leq N}\theta_3(n)\gg N$ (and we have an upper bound of the same order of magnitude from Selberg's sieve).  
 
 \begin{theorem}[Arbitrary linear configurations weighted by primes of the form $x^2+y^2+1$] \label{th: Linnik} Theorem \ref{th:linEqTwinPrimes} continues to hold with $\theta_3$ in place of $\theta_i$. Moreover, $C_3(\Psi)>0$ as soon as for every prime $p$ there exists $\mathbf{n}\in \mathbb{Z}^d$ such that $\prod_{i\in [t]}\psi_i(\mathbf{n})(\psi_i(\mathbf{n})+a(p))\not \equiv 0\PMod p$, where $a(p)=-1$ if $p\equiv -1 \PMod 4$ and $a(p)=0$ otherwise.
 \end{theorem}
 
 Theorem \ref{th:linEqTwinPrimes} has an immediate corollary to linear systems of equations within the Chen or bounded gap primes.
 
 \begin{corollary}[Linear equations in almost twin primes] \label{cor: lineq} Let $L_1,\ldots, L_t:\mathbb{Z}^d\to \mathbb{Z}$ be linear forms. Consider the linear system of equations 
 \begin{align}\label{equ13}
L_i(\mathbf{n})=0\quad \forall\,\, i\in 1,\ldots, t.
 \end{align}
 Suppose that the system has a solution in the positive real numbers. Then
 \begin{enumerate}[label=\upshape(\roman*)]
     \item The system \eqref{equ13} has a solution in $\mathcal{P}_{\textnormal{Chen}}^d$, provided that it has a solution in $A_p^d$ for every prime $p$, where $A_p=\{x\in \mathbb{Z}/p\mathbb{Z}:\,\,x(x+2)\not \equiv 0\PMod p\}$.
     \item The system \eqref{equ13} has a solution in $\mathcal{P}_{\mathcal{H}}^d$, provided that $0\in\mathcal{H}$ and it has a solution in $B_p^d$ for every prime $p$, where $B_p=\{x\in \mathbb{Z}/p\mathbb{Z}:\,\, \prod_{j\in[m]}(x+h_j)\not \equiv 0\PMod p\}$.
 \end{enumerate}
 \begin{proof}
 We may assume that each $L_i$ is primitive (i.e. its coefficients have no common prime factor) and that the linear forms are
 linearly independent (so $t\leq d$ and the system has full rank $t$). Since our system is homogeneous, we may assume that the span of the linear forms $L_i$ does not contain a linear form which has exactly two or one nonzero coefficients; indeed, otherwise there exists  
 $(i,j)\in [d]^2$ and coefficients $(a_i,a_j)\in\Z^2\setminus\{0,0\}$ such that $i\neq j$ and for any solution
 $(n_1,\ldots,n_d)\in\Z^d$ of the system we have $a_in_i-a_jn_j=0$.
 If $a_ia_j=0$ then $n_in_j=0$ and so the system has no solution in $A_p^d$ nor
 in $B_p^d$ (because $0\in\mathcal{H}$). 
 So we may assume that both $a_i$ and $a_j$ are nonzero and coprime. But then either $a_i=a_j=1$ and
  we may eliminate a variable to obtain an equivalent system with fewer variables, or there is a prime $p$ dividing $a_i$ but not $a_j$ (or vice versa).
We infer $n_in_j\equiv 0\PMod p$, so the system has no solution in $A_p^d$ nor
in $B_p^d$.

 Therefore, the lattice of integer solutions of the system has a multiplicity-free parametrization of the form $\Psi(\Z^{d-t})$, where $\Psi:\Z^{d-t}\rightarrow\Z^d$ is a system of linear forms. The system $\Psi$ has finite complexity, since no two forms of $\Psi$ are linearly dependent, owing to the assumption about the span of the $L_i$'s not containing linear forms with exactly one or two nonzero coefficients.
 
 Further, the local conditions (i) and (ii) imply that $C_1(\Psi)>0$ and $C_2(\Psi)>0$ respectively.
 We can then apply Theorem \ref{th:linEqTwinPrimes} to the convex body $K=\{\mathbf{x}\in \mathbb{R}^{d-t}:\,\, \Psi(\mathbf{x})\in [1,N]^d\}$  with $N\to \infty$, which satisfies $\Vol(K)\gg N^{d-t}$ since the original system of equations has a solution in the positive real numbers, to conclude the proof.
  \end{proof}

 \end{corollary}
 
 As we will see, our method works more generally for 
  $\mathcal{P}_{\mathcal{H},k}:=\{n\in \mathbb{N} : \abs{(n+\mathcal{H})\cap\mathbb{P}}\geq k\}$
  instead of $\mathcal{P}_{\mathcal{H}}$
whenever the admissible tuple $\mathcal{H}$ is sufficiently large in terms of $k$.

\subsection{Transference principles}
Given a finite complexity affine-linear map $\Psi : \Z^d\rightarrow\Z^t$ (often referred to as a system of finite complexity), a function $f : [N]\rightarrow\R_{\geq 0}$ (typically a weighted indicator function of a set of arithmetic interest) and a convex body $K\subset \R^d$,
the $\Psi$-count of $f$ in $K$ is given by
\begin{align*}
T_\Psi(f,K):=\sum_{\mathbf{n}\in K\cap\Z^d}\prod_{i\in [t]}f(\psi_i(\mathbf{n})).
\end{align*}
Thus Theorem \ref{th:linEqTwinPrimes} is about lower-bounding $T_\Psi(\theta_i,K)$
for $i\in\{1,2\}$.
Assume that 
\begin{align*}
\sum_{n\in [N]}f(n)\geq \delta N    
\end{align*}
for some $\delta >0$ and infinitely many integers $N$.
If $f$ takes its values in $[0,1]$
and if additionally 
$\Psi$ is a homogeneous translation-invariant linear system, a functional version of Szemer\'edi's theorem (see \cite[Theorem 11.1]{TV10}) allows one to prove that $T_\Psi(f,K)\gg_{\delta} \Vol(K)$. Now, if $f$ is instead unbounded (for example, the von Mangoldt function), the Green--Tao transference principle consists in approximating $f$ (assuming that $f\leq \nu$ for some ``pseudorandom measure'' $\nu$), by a bounded function $\tilde{f} : [N]\rightarrow [0,1]$ (called a dense model of $f$) in such a way that $T_\Psi(f,K)\approx T_\Psi(\tilde{f},K)$,
and invoking Szemer\'edi's theorem.

In our case, however, as we are interested in non-translation-invariant systems, we will need a different dense model and hence a different transference principle. To see the need for a different transference principle, consider the set $A=\{n\in \mathbb{N}:\,\, \{\sqrt{2}n^2\}\in [1/3,1/3+1/100]\}$ where $\{\cdot\}$ denotes the fractional part of a real number; any translation-invariant configuration can be found inside this set since it is dense by Weyl's criterion, but note that the configuration $(x,x+y,x+2y,y)$ does not occur in $A$ due to the relationship
$(x+2y)^2-2(x+y)^2+x^2-2y^2=0$.

In the case of the ternary Goldbach system 
$\Psi=\Psi_N:(n,m)\mapsto(n,m,N-n-m)$, the Matom\"aki--Shao transference principle \cite{matomaki-shao}  provides, under a Fourier-type condition, an approximating function $\tilde{f}$ to $f$ satisfying again $T_\Psi(f,K)\approx T_\Psi(\tilde{f},K)$, which is lower bounded
pointwise: $\tilde{f}(n)\gg_\delta 1$; however, this does not generalize to the higher complexity case, as the set $A$ above (which is Fourier uniform) demonstrates.

The proof of our main theorem produces more generally a lower bounded dense model for any system of finite complexity.
This results in a transference principle (Theorem \ref{th:transference}) of independent interest, which allows us to lower bound  $T_\Psi(f,K)$ as desired for any function $f$ which is bounded by a pseudorandom measure and dense in every ``higher order Bohr set'' (to be defined precisely later).
We then check these two conditions for our weighted indicator functions of almost twin primes, i.e. functions $\theta_1$ and $\theta_2$. This will follow by working out a reduction to the case of equidistributed higher order Bohr sets (Section \ref{sec:twin-prime-bohr}) and then adapting a Bombieri--Vinogradov theorem for primes twisted by nilsequences \cite{shao2020bombierivinogradov}, proven by the last two authors. We also note that our transference principle is slightly stronger than the Green--Tao transference principle even for translation-invariant systems in the sense that our pseudorandomness requirement is weaker (we do not need the correlation condition from \cite{GT-kAPinPrime}); we achieve this relaxation by applying work of Dodos and Kanellopoulos \cite{greek}.

\subsection{Acknowledgments}

PYB is grateful for the financial support and hospitality of the Max Planck Institute for Mathematics, Bonn.
While finishing up he  was supported by the joint FWF-ANR project Arithrand: FWF: I 4945-N and ANR-20-CE91-0006.
 XS was supported by the NSF grant DMS-1802224. JT was supported by a Titchmarsh Research Fellowship. We thank the anonymous referee for a thorough reading of the paper and many insightful comments and suggestions.

\section{Notation and preliminary definitions}

Throughout the paper, we will use bold face characters to denote vectors or tuples. The set of nonnegative reals is denoted by $\R_{\geq 0}$.
The expectation notation $\E_{x\in X}$ shall mean, for a finite set $X$, the averaging operator
$\frac{1}{\abs{X}}\sum_{x\in X}$.
Further we will use the Vinogradov notation $f\ll g$ or $g\gg f$ whenever two functions
$f$ and $g$ from $\N$ to $\R$ satisfy $|f|\leq Cg$ for some constant $C>0$; the parameters on which the implied constant $C$ depends may be indicated as subscripts.
The conjunction $f\ll g$ and $g\ll f$ will be denoted by $f\asymp g$.
For any assertion $A$, the number $1_A$ is 1 if $A$ is true and 0 if it is false. 
The indicator function of a set $X$ will also be denoted by $1_X$, which should generate no ambiguity.
As usual, we denote by $\Lambda, \varphi, d_k$ the von Mangoldt, Euler, and $k$-fold divisor functions, respectively. The greatest common divisor of two integers $n$ and $m$ will be denoted by $(n,m)$.
A vector or tuple of numbers will usually be denoted in bold font and its components in regular font.
Given an integer $N$, we denote the interval of integers $\{1,\ldots,N\}$ by $[N]$.
We will often identify the sets $[N]$ and $\Z/N\Z$, which we always implicitly do in the natural way (reduction modulo $N$). Thus a function $f$ defined on $[N]$ may naturally be seen as a function on $\Z/N\Z$ and vice versa. When $x$ is a positive real number, we define $[x]=[N]$ where $N=\lfloor x\rfloor$ is the integral part of $x$.

\subsection{Systems of linear forms}

Let $\Psi=(\psi_1,\ldots,\psi_t) : \Z^d \rightarrow \Z^t$
be a system of affine-linear forms. 
We first define a quantity that captures the  local behaviour of $\Psi$ modulo a prime $p$.

\begin{definition}[Local factors]
\label{def:locFac}
For each prime $p$, define the \emph{$p$-adic local factor} of $\Psi$ as
$$\beta_p(\Psi):=\E_{\mathbf{a}\in(\Z/p\Z)^d}\prod_{i\in[t]}\frac{p}{\varphi(p)}1_{\psi_i(\mathbf{a})\not\equiv 0 \PMod p}.$$ Observe that $\Psi$ is admissible as defined in footnote \ref{foot0} if and only if $\beta_p(\Psi)\neq 0$ for each $p$.
\end{definition}
We need to control the asymptotic behaviour of $\beta_p$ as $p$ approaches infinity,
whence the following easy variant of \cite[Lemma 1.3]{GT-linear-equations}.
\begin{lemma}
\label{locfac}
Let $\Psi=(\psi_1,\ldots,\psi_t) : \Z^d \rightarrow \Z^t$
be an admissible system of affine-linear forms, and let $p$ be a prime.
Suppose that there are $t_p$ linear forms among $\psi_1,\cdots,\psi_t$ modulo $p$ such that no two of them are linearly dependent over $\mathbb{F}_p$, and that $t_p$ is maximal for this property.
Then 
\[ \beta_p(\Psi)= \left(\frac{p}{\varphi(p)}\right)^{t-t_p} \left(1+O_{d,t}(p^{-2})\right). \]
\end{lemma}
\begin{proof}
Without loss of generality, assume that no two of $\psi_1,\cdots,\psi_{t_p}$ are proportional modulo $p$, and let $\Psi_p = (\psi_1,\cdots,\psi_{t_p})$. By maximality of $t_p$, $\beta_p(\Psi) = (p/\varphi(p))^{t-t_p} \beta_p(\Psi_p)$. Since no two of $\psi_i,\psi_j$ with $1 \leq i < j \leq t_p$ can be linearly dependent modulo $p$, one can  follow the proof of \cite[Lemma 1.3]{GT-linear-equations}
to conclude that $\beta_p(\Psi_p) = 1 + O_{d,t}(p^{-2})$.
\end{proof}

The next crucial condition on linear systems that we will require is the aforementioned \emph{finite complexity} condition, which we now quantify.
For an affine-linear form $\psi$, let $\dot{\psi}$ be its linear
part.
\begin{definition}[Complexity of a system]
For $A\subset [t]$, let $V_A$ be the set
of linear forms on $\Z^d$ generated by $\{\dot{\psi}_i\mid i\in A\}$.
Let $i\in[t]$. A system $\Psi$ of linear forms is said to be of {\em complexity at most $k$ at $i$}
if
there exists a partition of $[t]\setminus\{i\}$ into at most $k+1$ parts such that 
$\dot{\psi}_i\notin V_A$
for each part $A$ of the partition. 
It is said to be of \emph{complexity at most $k$} if it is of complexity at
most $k$ at any $i\in[t]$. The complexity is the minimum $k$ such that
the complexity is at most $k$, if there is any such $k\in \N$, in which case $k\leq t-2$.
Otherwise, it is said to be infinite.
\end{definition}
A convenient parametrization of a system of finite complexity is the normal form (cf. \cite[Definition 4.2]{GT-linear-equations}), which we now define; it facilitates
 multiple applications of the Cauchy--Schwarz inequality, yielding the
generalized von Neumann theorem \cite[Proposition 7.1]{GT-linear-equations} which we will use later.
Let $\mathbf{e_1},\ldots,\mathbf{e_d}$ be the canonical basis of $\Z^d$.
\begin{definition}[Normal form of a system]
The system $\Psi$ is in $s$-\textit{normal form at} $i\in[t]$
if
there exists a set $J_i\subset [t]\setminus\{i\}$ of cardinality at most $s+1$ such that
$\prod_{j\in J_i}\dot{\psi_i}(\mathbf{e_j})\neq 0$ whereas
for all $k\in[t]\setminus\{i\}$, we have
$\prod_{j\in J_i}\dot{\psi_k}(\mathbf{e_j})=0$.
The system $\Psi$ is in $s$-\textit{normal form} if it is in
$s$-normal form at each $i\in [t]$.
\end{definition}
One may assume that $s\leq t-2$.
Clearly, a system in $s$-normal form has complexity at most $s$.
Due to a simple linear-algebraic argument from \cite[Theorem 4.5]{GT-linear-equations},
 we may assume in Theorems
\ref{th:linEqTwinPrimes} and \ref{th: Linnik} that the system $\Psi$ is in $s$-normal form for some $s\leq t-2$. We summarise this reduction in the following proposition.
\begin{proposition}
\label{prop:newsyst}
Let $\Psi :\Z^d\rightarrow\Z^t$ be a system of
complexity $s$ and $K\subset [-N,N]^d$
be a convex body such that $\Psi(K)\in[1,N]^t$. Suppose that the homogeneous coefficients of $\Psi$ are bounded by $L$.
Then there exist an integer $d'=O_{d,t}(1)$, an integer $N'=O(N)$, a real number $L'=O(L^{O(1)})$, a convex body $K'\subset [-N',N']^d$ and a system 
$\Psi' : \Z^{d'}\rightarrow\Z^t$ of affine-linear forms
in $s$-normal form
such that 
for any $t$ functions $g_1,\ldots,g_t:  \Z\rightarrow\R$, we have
$$
\frac{1}{\Vol(K)}\sum_{\mathbf{n}\in\Z^d\cap K}\prod_{i\in[t]}g_i(\psi_i(\mathbf{n}))=\frac{1}{\Vol(K')}\sum_{\mathbf{n}\in\Z^{d'}\cap K'}\prod_{i\in[t]}g_i(\psi_i'(\mathbf{n})).
$$
Further, we have $\Vol(K')/N'^d\gg \Vol(K)/N^d$.
\end{proposition}
In this form, this proposition is essentially \cite[Proposition 2.5]{bienvenuThesis}.

\subsection{Gowers norms}

\begin{definition}[Gowers norms over abelian groups]
Let $Z$ be a finite abelian group.
Let $g:Z\rightarrow\C$ be a function and $k\geq 1$ an integer.
The \emph{Gowers $U^k$ norm} of $g$  is the expression
$$\nor{g}_{U^k(Z)}:=\Big(\E_{x\in Z}\E_{\mathbf{h}\in Z^k}\prod_{\boldsymbol\omega
\in\{0,1\}^k}\mathcal{C}^{\abs{\boldsymbol\omega}}g(x+\boldsymbol\omega\cdot \mathbf{h})\Big)^{2^{-k}},$$
where 
$\mathcal{C}$ is the conjugation operator and $\abs{\boldsymbol\omega}:=\sum_{i\in [k]}\omega_i$.
\end{definition}
For $k\geq 2$, this does define a norm, whereas 
$\nor{f}_{U^1(Z)}=\abs{\E_{x\in Z}f(x)}$. 
For every $k\geq 1$, we have $\nor{g}_{U^k(Z)}\leq \nor{g}_{U^{k+1}(Z)}$.

\begin{definition}[Gowers norms over intervals]
Given a function $f:\Z\to \mathbb{C}$ and an integer $N$,
we define its Gowers norm $\|f\|_{U^k[N]}$ over the interval $[N]$ as
\begin{align*}
\|f\|_{U^k[N]}:=\frac{\|f\cdot 1_{[N]}\|_{U^k(\mathbb{Z}/N'\Z)}}{\|1_{[N]}\|_{U^k(\mathbb{Z}/N'\Z)}},    
\end{align*}
where $N'>2N$ (say $N'=2N+1$ for concreteness) and $f\cdot 1_{[N]}$ and $1_{[N]}$ are extended to $\mathbb{Z}/N'\Z$ in the natural way. By \cite[Lemma A.2]{FrHo14}, this definition is independent of the choice of $N'$.
\end{definition}
Observe that if $N$ and $N'$ are two integers satisfying $\alpha N'\leq N\leq N'$ for some $\alpha >0$ and a 
function $f : [N]\rightarrow \mathbb{C}$ is extended to $\Z/N'\Z$ by setting $f(n)=0$ for $n\in \Z/N'\Z\setminus [N]$, then $\nor{f}_{U^s[N]}\asymp_{\alpha,s}\nor{f}_{U^s(\Z/N'\Z)}$ (see \cite[Lemma B.5]{GT-linear-equations}).
Another norm that we will need is the $L^p$ norm on $[N]$ equipped with the uniform probability measure,
thus 
$$\nor{f}_{L^p[N]}:=\left(\E_{x\in [N]}\abs{f(x)}^p \right)^{1/p},$$
for $p\geq 1$ a real number.
Finally, we define the dual Gowers norm over an interval by  $$\|f\|_{U^k[N]^{*}}:=\sup_{\|g\|_{U^k[N]}=1}|\mathbb{E}_{x\in [N]}f(x)g(x)|.$$

\subsection{Nilsequences}
\begin{definition}[Nilsequences] Let $G$ be a connected, simply-connected nilpotent Lie group,  and let $\Gamma\leq G$ be a lattice. A \emph{filtration} $G_{\bullet}=(G_i)_{i=0}^{\infty}$ on $G$ is an infinite sequence of subgroups of $G$ (which are also connected, simply-connected nilpotent Lie groups) satisfying
\begin{align*}
G=G_0=G_1\supset G_2\supset\cdots,     
\end{align*}
and such that the commutators satisfy $[G_i,G_j]\subset G_{i+j}$, and with the additional conditions that $\Gamma_i:=\Gamma\cap G_i$ is a lattice in $G_i$ for $i\geq 0$ and $G_{s+1}=\{\textnormal{id}\}$ for some $s$. 

The least such $s$ is called the \emph{degree} of $G_{\bullet}$ and the manifold $G/\Gamma$ is called a \emph{nilmanifold}.

A \emph{polynomial sequence} on $G$ (adapted to the filtration $G_{\bullet}$) is a sequence $g:\mathbb{Z}\to G$ satisfying the derivative condition
\begin{align*}
\partial_{h_1}\cdots \partial_{h_k}g(n) \in G_k     
\end{align*}
for all $k\geq 0$, $n\in \Z$ and $h_1,\ldots, h_k\in \mathbb{Z}$, where $\partial_h g(n):=g(n+h)g(n)^{-1}$ denotes the discrete derivative with shift $h$.

Now fix a nilmanifold $G/\Gamma$, a filtration $G_{\bullet}$ of degree $s$ and a polynomial sequence  
$g:\mathbb{Z}\to G$.
Further, assume that  the nilmanifold is equipped with a  Mal'cev basis  $\mathcal{X}$ (see \cite[Definition 2.1, Definition 2.4]{GT-nilsequence}; note that the Mal'cev basis depends on the fixed filtration, not only on the manifold).
A Mal'cev basis induces a right-invariant metric on $G$ (see \cite[Definition 2.2]{GT-nilsequence}), which descends to a right-invariant metric on $G/\Gamma$ and will usually be denoted by $d_{\mathcal{X}}(\cdot,\cdot)$.
If $F:G/\Gamma\to \mathbb{C}$ is Lipschitz with respect to the metric on $G/\Gamma$ induced by $\mathcal{X}$, it is bounded by compacity so
we let $$\nor{F}_{\mathrm{Lip}(\mathcal{X})}=\nor{F}_\infty+\sup_{\substack{x,y\in G/\Gamma\\x\neq y}}\frac{\abs{F(x)-F(y)}}{d_{\mathcal{X}}(x,y)},$$
and we call a sequence of the form $n \mapsto F(g(n)\Gamma)$ a \emph{nilsequence}. 
The degree of the nilsequence is then $s$, and it is said to be of \textit{complexity} 
at most $M$ if  each of the degree $s$, the dimension of $G/\Gamma$, the rationality of $\mathcal{X}$
and the Lipschitz constant of $F$ is at most $M$.
\end{definition}
We now introduce a class of nilsequences of bounded degree and controlled complexity.
\begin{definition}
\label{def:PsiDeltaK}
Let $s\geq 1$ and $\Delta, K \geq 2$. Define $\Xi_s(\Delta, K)$ to  be  the  collection of all nilsequences  $\xi:  \Z \to  \C$ of the form $\xi(n) = F(g(n)\Gamma)$, where

\begin{enumerate}

\item $G/\Gamma$ is a nilmanifold of dimension at most $\Delta$, equipped with a filtration $G_{\bullet}$  of degree $\leq s$ and a
$K$-rational  Mal'cev basis  $\mathcal{X}$ ;

\item $g: \Z \to G$ is a  polynomial sequence adapted to  $G_{\bullet}$;

\item $F: G/\Gamma \to  \mathbb{C}$ is a Lipschitz function satisfying $\|F\|_{\operatorname{Lip}(\mathcal{X})} \leq 1$.
\end{enumerate}
\end{definition}

\begin{definition}[Equidistributed nilsequences]
\label{def:PsiDeltaKeta}
For $\eta \in (0,1)$ and $x \geq 2$, define $\Xi_s(\Delta, K; \eta, x)$ to be the collection of those nilsequences $\xi \in \Xi_s(\Delta, K)$ of the form $\xi(n) = F(g(n)\Gamma)$ that fulfill the additional condition that the sequence $(g(n)\Gamma)_{1 \leq n \leq 10x}$ is totally $\eta$-equidistributed in $G/\Gamma$ (defined in~\cite[Definition 1.2]{GT-nilsequence}).
\end{definition}

We will loosely call such nilsequences $\eta$-\emph{equidistributed}. We caution that this notation is slightly different from~\cite{shao2020bombierivinogradov}, in that we do not require  $\int_{G/\Gamma} F = 0$ (where the integral is taken with respect to the unique Haar measure on $G/\Gamma$). We shall use $\Xi_s^0(\Delta, K; \eta, x)$  to denote the set of $\eta$-equidistributed nilsequences in $\Xi_s(\Delta, K; \eta, x)$ satisfying the additional condition that $\int_{G/\Gamma} F = 0$.

\section{A transference principle for arbitrary systems of linear equations}
\label{sec:reg}
A fundamental notion related to transference principles is that of \emph{pseudorandom measures},
the basic philosophy being that if a function is bounded by such a measure, it behaves as if
it was bounded by 1.
\begin{definition}
A  function $\nu: \Z/N\Z\rightarrow \R_{\geq 0}$ is said to satisfy the $(d_0,t_0,L_0,\varepsilon)$-\textit{linear forms conditions} if it satisfies the following. Let $1\leq d\leq d_0$ and $1\leq t\leq t_0$. For every finite complexity system of affine-linear forms $\Psi=(\psi_1,\ldots,\psi_t) : \Z^d\rightarrow \Z^t$ with linear coefficients bounded by $L_0$ in modulus,  the following estimate holds :
\begin{equation}
\label{defeqlinforms}
\abs{\E_{\mathbf{n}\in (\Z/N\Z)^d}\prod_{i\in [t]} \nu(\psi_i(\mathbf{n}))-1}\leq\varepsilon.
\end{equation}
If it satisfies the $(M,M,M,\varepsilon)$-linear forms conditions, it is said to be $(M,\varepsilon)$-pseudorandom.
\end{definition}
Observe that $\nor{\nu-1}_{U^k(\Z/N\Z)}=O(\varepsilon^{1/2^k})$ as soon as
$\nu$ satisfies the $(k+1,2^k,1,\varepsilon)$-linear forms conditions, and that
the constant coefficients of $\Psi$ are unrestricted.
Note that our definition
is less restrictive than that of Green and Tao \cite{GT-linear-equations}, since
we do not require the so-called correlation condition.

The aim of this section is to prove the following theorem.

\begin{theorem}[Transference principle for linear systems]
\label{th:transference}
Let $t,d,L, s \geq 1$ be integers, and let $\delta, \eta>0$ be real numbers. 
Then there exist constants $M\geq 1$ depending on $d,t,L$ and $Y,\varepsilon >0$ depending on $d,t,L,\delta$ such that the following holds.
Let $\Psi=(\psi_1,\ldots,\psi_t):\Z^d\rightarrow\Z^t$ be a system of affine-linear forms of complexity $s$ whose homogeneous parts have coefficients bounded by $L$.
Let $N$ be a large enough prime  and $\alpha$ be small enough (both in terms of $t,d,L,\eta$).
Let $K\subset [-N,N]^d$ be a convex body satisfying
$\mathrm{Vol}(K)\geq \eta N^d$ and
 $\Psi(K)\subset [1,N]^t$. Lastly, for each $i\in[t]$, let $\lambda_i:[N]\to \mathbb{R}_{\geq 0}$ be a function. Assume that the following additional hypotheses hold.
\begin{enumerate}[label=\upshape(\roman*)]
    \item There exists an $(M,\alpha)$-pseudorandom measure $\nu : \Z/N\Z\rightarrow\R_{\geq 0}$
    such that 
    $$\lambda_i(n)\leq \nu(n)\quad \textnormal{for each}\,\, i\in [t]\,\, \textnormal{and}\,\,n\in [N]$$
    where we identify $[N]$ and $\Z/N\Z$ in the natural way.
    \item Each function $\lambda_i$ is \emph{dense in higher order Bohr sets} in the sense that
    \begin{align*}
     \mathbb{E}_{n \in [N]}\lambda_i(n)\xi(n)\geq \delta  \mathbb{E}_{n \in [N]}\xi(n)   
    \end{align*}
    for every nilsequence $\xi:\mathbb{Z}\to [0,1]$ of degree $\leq s$ and of complexity at most $Y$ that satisfies $\mathbb{E}_{n\leq N}\xi(n)\geq \varepsilon $.
\end{enumerate}

Then we have
\begin{align}\label{equ1}
\mathbb{E}_{\mathbf{n}\in K\cap\Z^d}\lambda_1(\psi_1(\mathbf{n}))\lambda_2(\psi_2(\mathbf{n}))\cdots \lambda_t(\psi_t(\mathbf{n}))\geq 0.99\delta^{t}.   
\end{align}

\end{theorem}
We now present the tools which will enable us to prove this.
We need a notion of higher order Bohr sets, which are roughly speaking  sets that are approximated by level sets of nilsequences to any given accuracy. The following  definitions of $s$-measurable sets and $s$-factors are from \cite[Section 2]{GT10corr}.

\begin{definition}[$s$-measurable sets]\label{def:bohr}
Let $s \geq 1$ and let $\Phi: \R \to \R$ be a growth function. A subset $E \subset [N]$ is called \emph{$s$-measurable} with growth function $\Phi$ if, for any $M \geq 1$, there exists a degree $\leq s$ nilsequence $\xi: \Z \rightarrow [0,1]$ of complexity at most $\Phi(M)$ such that $\|1_E - \xi\|_{L^2[N]} \leq 1/M$.
\end{definition}

\begin{definition}
If $\CB$ is a partition of $[N]$, we call its parts $E\in \CB$ \emph{atoms}. The \emph{conditional expectation} of a function $f : [N]\rightarrow\R$ with respect to $\CB$ is
the function $\E[f|\CB]$ which is constant on each atom, equal to the average of $f$ on the atom.
\end{definition}
\begin{definition}[$s$-factors]\label{def:s-factor}
Let $s \geq 1$ and let $\Phi: \R \to \R$ be a function. 
A partition $\CB$ of $[N]$ is called an \emph{$s$-factor} of complexity at most $M$ and growth function $\Phi$
if $\CB$ contains at most $M$ atoms and each atom is 
$s$-measurable of growth function $\Phi$.
\end{definition}

The following two propositions will be important in our proof of Theorem \ref{th:transference}.
The first one is the weak regularity lemma\footnote{It was discovered recently that the reference \cite{GT10} contains a slight error. See the arXiv version \cite{GT10corr} for details and a correction.
Nevertheless, the regularity lemma part of that reference is unaltered, only the counting lemma (and what depends on it) was not entirely correct.}  proved in \cite[Corollary 2.6]{GT10}.
\begin{proposition}[Weak regularity lemma]\label{prop:weak-reg}
Let $s \geq 1$ and $\varepsilon > 0$.
Let $f: [N] \to \R$ be a function with $|f(n)| \leq 1$ pointwise. There exists a  function $\Phi : \R \to \R$ depending only on $s,\ee$ and an $s$-factor $\CB$ of complexity $O_{s,\ee}(1)$ and growth function $\Phi$ such that  $\|f - \E(f \vert \CB)\|_{U^{s+1}(\Z/N\Z)} \leq \varepsilon$.
\end{proposition}
In the just cited the reference, the Gowers norms are interval Gowers norms, but this makes no difference since
the $U^{s+1}(\Z/N\Z)$ and $U^{s+1}[N]$ norms are equivalent on bounded functions (see \cite[Lemma A.4]{FrHo14} for instance).
We also state the dense model theorem from the work of Dodos and Kanellopoulos
\cite[Corollary 4.4]{greek}. 
\begin{proposition}[Dense model theorem]
\label{denseModel}
Let $s \geq 1$ and let $ Z$ be a finite abelian group. Let $0<\eta\leq 1$. Suppose that $\nu : Z\rightarrow\R_{\geq 0}$
satisfies $\nor{\nu-1}_{U^{2s}(Z)}\leq \eta$, and that $f : Z\rightarrow \R$  is a function such that $ \abs{f(n)}\leq \nu(n)$ pointwise. Then we may decompose $f = f_1+ f_2$, where
$\sup_{n\in Z} |f_1(n)|\leq 1$
and $\| f_2\|_{U^{s}(Z)}= o_{\eta\rightarrow 0;s}(1)$.

Further, if $f$ is nonnegative, so is $f_1$.
\end{proposition}
We note that this version of the dense model theorem has weaker hypotheses
 than the one in \cite{GT-linear-equations} (it does not require the so-called correlation condition), a fact that will be important for us.
 A dense model for arithmetic progressions was also achieved without correlation conditions by Conlon, Fox and Zhao \cite{CFZ1,CFZ2}, but their dense model is not as strong as we need since it is not close in the Gowers norms topology to the function to be modeled.
 
 Finally, we state a version of the generalised von Neumann theorem \cite[Proposition 7.1']{GT-linear-equations}.
 \begin{proposition}[Generalised von Neumann theorem]
\label{vonNeumann}
Let $t,d,L,s$ be positive integer parameters. 
Let $\delta,\varepsilon$ be in $(0,1)$ and $N\geq 1$.
Then there is a positive constant $D$, depending on $t,d$ and $L$ such that
the following holds. Let $\nu : \Z/N\Z
\rightarrow \R_{\geq 0}$ be a $(M,\varepsilon)$-pseudorandom measure, and suppose
that $f_1,\ldots,f_t : \Z/N\Z\rightarrow\R$
are functions with $\abs{f_i(x)}\leq \nu (x)$
for all $i\in[t]$ and $x\in \Z/N\Z$.
Suppose that $\Psi=(\psi_1,\ldots,\psi_t)$ is a system of affine-linear forms in $s$-normal form whose linear coefficients are bounded by $L$. 
Let $K'\subset (\Z/N\Z)^d$ be identified with $K\cap\Z^d$ where $K\subset [-N/4,N/4]^d$ is a convex set.
Finally, suppose that 
\begin{equation}
\min_{1\leq j\leq t}\nor{f_j}_{U^{s+1}[N]}\leq \delta.
\label{uniformity}
\end{equation}
Then we have
$$
\E_{\mathbf{n}\in (\Z/N\Z)^d}1_{K'}(\mathbf{n})\prod_{i\in[t]}f_i(\psi_i(\mathbf{n}))=o_{\delta\rightarrow 0}(1)
+o_{N\rightarrow\infty;\delta}(1)+o_{\varepsilon\rightarrow 0;\delta}(1)
$$
where the $o(1)$ terms may also depend on $d,t,L$.
\end{proposition}
In the cited reference, the parameter $\varepsilon$ is itself $o_{N\rightarrow\infty}(1)$ but we make
it independent here, whence the slightly different statement.

We are now ready to state and prove a crucial lemma.
\begin{proposition}[Decomposition into a uniformly lower bounded and Gowers uniform components]
\label{prop:crux}
Let $s \geq 1$ and let $ N \geq 1$ be an integer. 
Let $\delta, \varepsilon, \rho$
be real constants in the interval $(0,1)$.
Then there exist quantities $\iota\in (0,1),Y>0,\eta\in(0,1)$ depending only on $s,\varepsilon,\rho,\delta$ such that the following holds.
Suppose that $\nu : [N]\rightarrow \R_{\geq 0}$
satisfies $\nor{\nu-1}_{U^{2s+2}(\Z/N\Z)}\leq \eta$, where
we naturally identify $[N]$ with $\Z/N\Z$.  Let $f : [N]\rightarrow \R_{\geq 0}$  be a function such that $ f(n)\leq \nu(n)$ pointwise. 
Further, suppose that
\begin{align*}
     \mathbb{E}_{n \in [N]}f(n)\xi(n)\geq \delta  \mathbb{E}_{n \in [N]}\xi(n)   
    \end{align*}
    for every nilsequence $\xi:\mathbb{Z}\to [0,1]$ of degree $\leq s$ and of complexity at most $Y$ that satisfies $\mathbb{E}_{n\leq N}\xi(n)\geq \iota N$.

Then there exists a decomposition
$f=f_3+f_4$ where
 $f_3\geq (1-\rho)\delta$ pointwise and
 $\nor{f_4}_{U^{s+1}(\Z/N\Z)}\leq \varepsilon$.
\end{proposition}
\begin{proof}
Without loss of generality, we may assume that $\varepsilon$ is small enough in terms of $\delta$ and $\rho$. 
Let $\varepsilon'\leq \varepsilon/3$ be a sufficiently small constant, to be determined later.
Applying Proposition~\ref{denseModel}, we may write $f=f_1+f_2$
where $f_1$ takes its values in $[0,1]$ and $\nor{f_2}_{U^{s+1}(\Z/N\Z)}\leq \varepsilon'$.
Then using Proposition \ref{prop:weak-reg} on $f_1$,
we decompose $f_1=h+g$
where $h=\E[f_1|\CB]$ and  $\CB$ is an $s$-factor of complexity $O_{s,\varepsilon}(1)$ and growth function $\Phi$,
and $\|g\|_{U^{s+1}(\Z/N\Z)}\leq \varepsilon/3$.
The growth function $\Phi$ of $\CB$
depends only on $\varepsilon$ and $s$.

Note that 
$h=\sum_{E\in\CB}c_E1_E$
where $c_E=\frac{1}{\abs{E}}\sum_{n\in [N]}f_1(n)1_E(n)$
for any atom
$E$ of $\CB$.
Fix 
$c=(\varepsilon/3)^{2^{s+1}}/\abs{\CB}$, so $c^{-1}=O_{s,\varepsilon}(1)$.
We effect the splitting
$$h=\sum_{E\in\CB}(c_E+\delta 1_{\abs{E}< cN})1_E-\sum_{\substack{E\in\CB\\
\abs{E}< cN}}\delta1_E.$$
We denote by $h_1$ the first sum and $h_2$ the second one.

Since $\delta\in [0,1]$, we see
that $\nor{h_2}_{L^1[N]}\leq c|\CB| =(\varepsilon/3)^{2^{s+1}}$.
Crudely estimating by the triangle inequality, this implies that $\nor{h_2}_{U^{s+1}(\Z/N\Z)}\leq \varepsilon/3$.
Now write $f_3=h_1$ and $f_4=g+f_2+h_2$. By the triangle inequality for Gowers norms, we have
\begin{align*}
\nor{f_4}_{U^{s+1}(\Z/N\Z)}\leq 3\cdot\varepsilon/3=\varepsilon.
\end{align*}
Our aim is then to show that 
\begin{align}\label{eq: ce}
c_E\geq (1-\rho)\delta,\quad \textnormal{whenever}\quad |E|\geq cN,    
\end{align}
after which $f_3=h_1\geq (1-\rho)\delta$ pointwise follows. 

Fix a large enough constant $M>0$ in terms of $c,\delta, \rho$ (explicitly, we may take $M=4/(c\delta \rho)$) 
and an atom $E\in \CB$ satisfying $\abs{E}\geq cN$.
By definition of an $s$-factor, 
we may write $1_E=\xi+g_{\textnormal{sml}}$
where $\nor{g_{\textnormal{sml}}}_{L^2[N]}\leq 1/M$
and $\xi$ (depending on $E$) is a nilsequence of degree at most $s$ whose complexity is  bounded by $Y:=\Phi(M)=O_{s,\varepsilon,\rho,\delta}(1)$.
By the Cauchy--Schwarz inequality,
we have $|\sum_{n\in [N]}f_1(n)g_{\textnormal{sml}}(n)|\leq N/M$.
Therefore,
\begin{align*}
c_E\geq \frac{1}{\abs{E}}\sum_{n\in [N]}f_1(n)\xi(n)-1/(cM).    
\end{align*}
We now recall that $f=f_1+f_2$, so that \eqref{eq: ce} follows once we show that
\begin{align}\label{eq: lowerbound}
\sum_{n\in [N]}f(n)\xi(n)\geq \delta \abs{E}(1-1/(cM))\quad \textnormal{and}\quad  \Big|\sum_{n\in [N]}f_2(n)\xi(n)\Big|\leq \rho\delta \abs{E}/2,   
\end{align}
upon setting $M=4/(c\delta\rho)$.
First we bound the correlation of $f_2$ and $\xi$.
Since $f_2$ has small $U^{s+1}$ norm, it would be
convenient to replace $\xi$ by a function of bounded
dual $U^{s+1}$ norm.
To achieve this, 
we invoke\footnote{\label{foot1}In the cited reference,
written at a time where the theory of nilsequences was just emerging, the result is stated  for linear nilsequences. However, nowadays we know that any polynomial nilsequence may be realized as a linear one, see \cite[Appendix C]{GTZ}. Also the result is stated in terms of interval Gowers norms, but the proof naturally yields $\nor{\xi_1}_{U^{s+1}(\Z/N\Z)^*}\leq K$ first as it moves from intervals to cyclic groups.}  \cite[Proposition 11.2]{GT-linear-equations},
which yields for any $\kappa>0$ a splitting
$\xi=\xi_1+\xi_2$
where $\nor{\xi_1}_{U^{s+1}(\Z/N\Z)^*}\leq K$ for some $K=O_{\kappa,Y}(1)$,  while $\nor{\xi_2}_{\infty}\leq \kappa$.
We infer that
\begin{align}\label{eq: lowerbound1}
\Big|\sum_{n\in [N]}f_2(n)\xi_1(n)\Big|\leq \nor{f_2}_{U^{s+1}(\Z/N\Z)}\nor{\xi_1}_{U^{s+1}(\Z/N\Z)^*}N\leq \varepsilon' KN.
\end{align}
Further, since $\abs{f_2}\leq \nu +1$ pointwise,
\begin{align}\label{eq: lowerbound2}
\Big|\sum_{n\in [N]}f_2(n)\xi_2(n)\Big|\leq 
\nor{\xi_2}_{\infty}\sum_{n\in [N]}(\nu(n)+1)\leq (2+\lvert\E_{n\in[N]}(\nu(n)-1)\rvert)\nor{\xi_2}_{\infty}N.
\end{align}
Recall that $\lvert\E_{n\in[N]}(\nu(n)-1)\rvert)=\nor{\nu-1}_{U^1(\Z/N\Z)}\leq \nor{\nu-1}_{U^{2s+2}(\Z/N\Z)}\leq\eta<1$.
We conclude that $\Big|\sum_{n\in [N]}f_2(n)\xi_2(n)\Big|\leq 3\kappa N$.

Now, if we choose $\kappa=c\rho\delta/12$ and $\varepsilon'=\rho\delta c/(4K)$ (thus $(\varepsilon')^{-1}=O_{\rho,s,\delta,\varepsilon}(1)$), we have $\varepsilon' K+3\kappa<c\rho\delta/2$, so by combining \eqref{eq: lowerbound1} and \eqref{eq: lowerbound2} with the fact that $N\leq c^{-1}\abs{E}$, we obtain the required bound \eqref{eq: lowerbound} for $f_2$.

It remains to be shown that the correlation of $f$ and $\xi$ obeys the lower bound in \eqref{eq: lowerbound}. By the definition of $\xi$, we have \begin{align*}
\sum_{n\in [N]}\xi(n)= \abs{E}-\sum_ng_{\textnormal{sml}}(n)\geq \abs{E}(1-1/(cM))\geq |E|/2\geq cN/2,    
\end{align*}
so recalling the ``denseness in higher order Bohr sets'' hypothesis of Proposition \ref{prop:crux} (letting $\iota=c/2$ there) and the fact that the complexity 
of $\xi$ is at most $Y$, the desired estimate \eqref{eq: lowerbound} follows. This was enough to complete the proof.
\end{proof}

Now we may prove Theorem \ref{th:transference}. 
\begin{proof}[Proof of Theorem \ref{th:transference}]
In view of Proposition \ref{prop:newsyst}, we may assume that the system $\Psi$
is in $s$-normal form. Also, upon replacing $N$ by $4N$ (and therefore $\eta$ by $4^{-d}\eta$), we may assume that $K\subset [-N/4,N/4]$.
Fix $\kappa>0$ small enough (to be determined later). Let $\rho$ be small enough in terms of $t$ (say $\rho=1/(10000t)$). By hypothesis (i), if $N$ is large enough, we may apply Proposition \ref{prop:crux}, thus obtaining a decomposition $\lambda_i=\lambda_i^{(1)}+\lambda_i^{(2)}$ for each $i\in[t]$
where $\lambda_i^{(1)}\geq (1-\rho)\delta$ pointwise and
$\nor{\lambda_i^{(2)}}_{U^{s+1}(\Z/N\Z)} \leq\kappa$.
Inserting this decomposition in the left-hand side of
\eqref{equ1}, we obtain a splitting of the average into
$2^t$ terms:
\begin{align*}
\mathbb{E}_{\mathbf{n}\in K\cap\Z^d}\prod_{j=1}^t \lambda_j(\psi_j(\mathbf{n}))=\sum_{a_1,a_2,\ldots, a_t\in \{1,2\}}\mathbb{E}_{\mathbf{n}\in K\cap\Z^d}\prod_{j=1}^{t}\lambda_j^{(a_j)}(\psi_j(\mathbf{n})).   
\end{align*}
One of the $2^t$ terms involves only the functions
$\lambda_i^{(1)}$; since $\lambda_i^{(1)}$ is pointwise lower bounded by $(1-\rho)\delta$, this term is at least $$(1-\rho)^{t}\delta^t\geq 0.999\delta^{t},$$
since $\rho=1/(10000t)$.  
Any other term involves at least one copy of a uniform
function $\lambda_i^{(2)}$. 
Let $\mathbf{a}\in \{1,2\}^t\setminus\{(1,\ldots,1)\}$.
Since $K\subset [-N/4,N/4]^d$ and $\Psi(K)\subset [1,N]^t$, one may identify $K$ with a subset of $(\Z/N\Z)^d$, which we also denote by $K$, and write
\begin{equation}
\label{eq:transferCyclic}
\mathbb{E}_{\mathbf{n}\in K\cap\Z^d}\prod_{j=1}^{t}\lambda_j^{(a_j)}(\psi_j(\mathbf{n}))
=\mathbb{E}_{\mathbf{n}\in(\Z/N\Z)^d}1_K(\mathbf{n})\prod_{j=1}^{t}\lambda_j^{(a_j)}(\psi_j(\mathbf{n})).
\end{equation}
According to Proposition \ref{vonNeumann} (for which we need $M$ to be sufficiently large in terms of
$d,t,L$ and the fact that $\Psi$ is in $s$-normal form),
the right-hand side of equation \eqref{eq:transferCyclic} is bounded by $o_{N\rightarrow \infty;\kappa}(1)+o_{\alpha\rightarrow 0;\kappa}(1)+o_{\kappa\rightarrow 0}(1)$. Therefore, choosing first $\kappa$ appropriately, and then $N$ sufficiently large and $\alpha$ sufficiently small, we conclude the proof.
\end{proof}
The rest of the paper is devoted to establishing the
hypotheses (i) and (ii) of Theorem~\ref{th:transference}
for the functions $\theta_1$ and $\theta_2$ (and $\theta_3$ in Section \ref{sec: mainthm}); we start with hypothesis (i).

\section{\texorpdfstring{$W$}{W}-trick and Pseudorandom majorants}\label{sec: w}
\subsection{\texorpdfstring{$W$}{W}-trick}
We wish to apply Theorem \ref{th:transference} to prove our main theorem, but an initial problem is that the indicator functions of almost twin primes are not bounded by a pseudorandom majorant, as they are biased modulo small primes. We will first have
to remove these biases modulo small primes to obtain a pseudorandomly majorized function. 

We introduce the general framework we will work with in this section.
Let $w$ be an integer and $W=W(w)=\prod_{p\leq w}p$.
Let $\rho\in (0,1)$, $r\geq 1$, and let $\mathcal{H}=\{h_1,\ldots,h_r\}\subset\N$ be a set of $r$ pairwise distinct integers
and let $\theta=\theta_\mathcal{H} :\N\rightarrow\R_{\geq 0}$ be any function supported on the set 
$$\{n\in \mathbb{N}:\,\, p\mid n+h_j\implies p>w\,\, \textnormal{for all}\,\, j\in [r]\}$$
and satisfying the upper bound $$\theta(n)\leq \log^r(n+2)$$ 
for all $n\geq 0$.
Observe that the functions $\theta_1$ and $\theta_2$ our main theorem deals with have these properties (with $r=2$ in the case of $\theta_1$ and $r=m$ in the case of $\theta_2$).
Given integers $q>0$ and $b$, let 
\begin{align}\label{eq: theta}
\theta_{q,b}(n):=\left(\frac{\varphi(q)}{q}\right)^r\theta(qn+b).
\end{align}

\begin{proposition}[Reduction to $W$-tricked sums] \label{prop: wtrick} Let the notation be as above. Also let $\eta>0$, $\gamma>0$, $N,L,d,t\geq 1$. 
Let $\Psi=(\psi_1,\ldots,\psi_t):\mathbb{Z}^d\to \mathbb{Z}^t$  be a finite complexity system of affine-linear forms.
Suppose that $\Psi_{\mathcal{H}}:=(\psi_i+h_j)_{i\in[t],j\in[r]}$ is admissible
 and that the linear coefficients of $\Psi$ as well as the elements of $\mathcal{H}$ are bounded by $L$.
Let $K\subset [-N,N]^d$ be a convex body satisfying
$\mathrm{Vol}(K)\geq \eta N^d$ and
 $\Psi(K)\subset [1,N]^t$.
 Suppose that $\theta$ satisfies
\begin{equation}
\label{eq:Wtricked}
\sum_{\substack{\mathbf{n}\in 
\Z^d\\W\mathbf{n}+\mathbf{a}\in K}}\prod_{i=1}^t\theta_{W,c_i(\mathbf{a})}(\psi_i'(\mathbf{n}))\geq \gamma W^{-d}\Vol(K),
\end{equation}
for each $\mathbf{a}\in A$, where
\begin{align*}
A=A_{\Psi,\mathcal{H}} 
=\{\mathbf{a}\in[W]^d: \forall (i,j)\in[t]\times[r]\quad (\psi_i(\mathbf{a})+h_j,W)=1\}
\end{align*}
and for each $i\in [t]$, the integer $c_i(\mathbf{a})\in [W]$ and the form $\psi'_i : \Z^d\rightarrow\Z$ 
are uniquely defined by the relation $\psi_i(W\mathbf{n}+\mathbf{a})=W\psi'_i(\mathbf{n})+c_i(\mathbf{a})$.
Then, provided that $w$ is large enough in terms of $d,t,L$, we have
\begin{equation}
 \label{eq:linEqTwinPrimes2}
 \sum_{\mathbf{n}\in K\cap \Z^d}\prod_{i=1}^t\theta(\psi_i(\mathbf{n}))\geq \frac{\gamma}{2}\cdot \prod_{p}\beta_p\cdot \Vol(K),
 \end{equation} 
 where the local factors $\beta_p=\beta_p(\Psi_{\cH})$ are as defined in Definition \ref{def:locFac}.
\end{proposition}

\begin{proof}
We write
$$
\Z^d\cap K =\bigcup_{\mathbf{a}\in [W]^d} (\Z^d\cap (W K_\mathbf{a}+\mathbf{a})) ,
$$
where $$K_\mathbf{a}:=\{\mathbf{x}\in \R^d: W \mathbf{x}+\mathbf{a}\in K\}$$
is again a convex body.
Putting $$F(\mathbf{n}):=\prod_{j=1}^t\theta(\psi_j(\mathbf{n}))$$
we can write the left-hand side of \eqref{eq:linEqTwinPrimes2} as
\begin{equation}
\label{sumovercongclasses}
\sum_{\mathbf{n}\in \Z^d\cap K}F(\mathbf{n})=\sum_{\mathbf{a}\in [W]^d}\sum_{\mathbf{n}\in 
\Z^d\cap K_\mathbf{a}}F(W\mathbf{n}+\mathbf{a}).
\end{equation}

We note that if $\psi_i(\mathbf{a})+h_j$ is not coprime to $p$ for some $i\in [t]$, some $j\in[r]$
 and some prime $p\leq w$, 
 then for each $\mathbf{n}\in K_\mathbf{a}\cap\Z^d$ we have
$F(W\mathbf{n}+\mathbf{a})=0$: indeed, in that case, the integer
$\psi_i(W\mathbf{n}+\mathbf{a})+h_j$ has a prime factor $p\leq w$, hence does not belong to the support of $\theta$. Thus, the residues $\mathbf{a}$  which 
bring a nonzero contribution to the right-hand side of
\eqref{sumovercongclasses}
are all mapped by
$\Psi$ to tuples $(b_1,\ldots,b_{t})$ for which $b_i+h_j$ is coprime to $W$ for all $i\in[t],j\in[r]$. 

Recalling the definitions of $A=A_{\Psi,\mathcal{H}}$, the integers $c_i(\mathbf{a})$, the forms $\psi_i'$ and $\theta_{W,b}$,
we can then rewrite equation \eqref{sumovercongclasses} as
\begin{equation}
\label{sumovercongclasses2}
\sum_{\mathbf{n}\in \Z^d\cap K}F(\mathbf{n})=\left(\frac{W}{\varphi(W)}\right)^{rt}\sum_{\mathbf{a}\in A}\sum_{\mathbf{n}\in 
\Z^d\cap K_\mathbf{a}}\prod_{i=1}^t\theta_{W,c_i(\mathbf{a})}(\psi_i'(\mathbf{n})).
\end{equation}
By our assumption \eqref{eq:Wtricked}, we have
\begin{equation}
\label{eq:Wtricked2}
\sum_{\mathbf{n}\in 
\Z^d\cap K_\mathbf{a}}\prod_{i=1}^t\theta_{W,c_i(\mathbf{a})}(\psi_i'(\mathbf{n}))\geq \gamma W^{-d}\Vol(K)=\gamma \Vol(K_\mathbf{a})
\end{equation}
for each $\mathbf{a}\in A$, so to obtain the conclusion \eqref{eq:linEqTwinPrimes2} it suffices to prove that
\begin{equation}
\label{eq:Apsi}
\left(\frac{W}{\varphi(W)}\right)^{rt}\abs{A_{\Psi,\mathcal{H}}}\geq \frac{1}{2}W^d\prod_{p}\beta_p.
\end{equation}
Note that by the Chinese remainder theorem we have 
$$\left(\frac{W}{\varphi(W)}\right)^{rt}\abs{A_{\Psi,\mathcal{H}}}=W^d\prod_{p\leq w}\beta_p.$$

Lemma \ref{locfac} implies that $\beta_p=1+O_{d,t,L}(p^{-2})$ whenever $\abs{\cH \PMod{p}}=r$ (which is the case whenever $w>H$), and $\beta_p>0$ for any $p$ since $\Psi_\mathcal{H}$ is an admissible system.
Therefore $\prod_p\beta_p$ is convergent and if $w>H$ we have $\prod_{p\leq w}\beta_p=(1+O_{d,t,L}(1/w))\prod_p \beta_p$.
Taking $w$ large enough in terms of $d,t,L$, this concludes the proof of Proposition \ref{prop: wtrick}.
\end{proof}
In fact Lemma \ref{locfac} implies that $\beta_p=1+O(p^{-1})$ and $\beta_p=1+O(p^{-2})$ except when $p\mid \prod_{i,j}(h_i-h_j)$. 
Combining this with the fact that, for any integer $q$ having $z$ prime factors,
we have 
$$\prod_{\substack{w<p\\p\mid q}}(1+O(p^{-1}))\leq \prod_{w<p<w+z}(1+O(p^{-1}))=O(\log\log q/\log w),$$ we infer
$
\prod_{p> w}\beta_p\ll O(\log\log H/\log w)
$
so the weaker hypothesis $H\leq \exp(w^{O(1)})$ could suffice instead of $w>H$ at the cost of replacing $1/2$ by a worse constant.

Also we note that the system $\Psi'$ introduced above differs from $\Psi$ only in the constant term,
and so it is of finite complexity whenever $\Psi$ is.

\subsection{Pseudorandom majorants}\label{sub: pseudo}

In order to prove Theorem \ref{th:linEqTwinPrimes}, it remains to establish the lower bound \eqref{eq:Wtricked}
when $\theta$ is either $\theta_1$ or $\theta_2$, which we will do by invoking Theorem \ref{th:transference}. In order to appeal to this theorem, we need to supply a pseudorandom majorant for the function
$\theta_{W,b}$ where $b$ is coprime to $W$. The only properties of $\theta$ that we need for this construction are that it is supported on the set
$$\{n\in \mathbb{N}:\,\, p\mid n+h_j\implies p>n^{\rho}\,\, \textnormal{for all}\,\, j\in [r]\}$$
and satisfies $0\leq \theta(n)\leq \log^r(n+2)$ (and these are satisfied for $\theta_1,\theta_2$ with $r=2,m$, respectively). Let
\begin{align}\label{equ20}
B_{\mathcal{H}}:=\{b\in \N:\forall j\in [r]\,\,\, (b+h_j,W)=1\}.
\end{align}
The next proposition provides us with a pseudorandom majorant. 
\begin{proposition}[Pseudorandom majorants]\label{prop: pseudo}
Let $M\geq 1$ and $s$ be integers. Let $\epsilon>0$.
Assume that $N$ and $w$ are large enough in terms of $(M,\varepsilon)$
and satisfy $w\leq\log\log N$.
Let $\mathbf{b}=(b_1,\ldots,b_s)$ be in $B_{\mathcal{H}}^s$ satisfy $\abs{b_i-b_j}\leq M$ for any $(i,j)\in [s]^2$. 
Suppose also that $\theta$ is as above.
Then there is an $(M,\varepsilon)$-pseudorandom measure $\nu_{\mathbf{b}} :\Z/N\Z\rightarrow\R_{\geq 0}$
and a constant $c\in (0,1)$ depending on $M$ only
such that $$\theta_{W,b_1}(n)+\cdots+\theta_{W,b_s}(n)\ll_{M} \nu_{\mathbf{b}}(n)$$
for all $n\in [N^c,N]\subset \Z/N\Z$.
\end{proposition}

The fact that the bound does not necessarily hold on the full interval $[N]$ is not a serious restriction, as we may impose $\theta$ to be supported in $[N^c,N]$ without  changing the left-hand side
of \eqref{eq:Wtricked} and \eqref{eq:linEqTwinPrimes2} by more than $O(N^{cd}\log^{O(1)}N)=o(N^d)$.
Zhou \cite{kAP-Chen} and Pintz \cite{kAP-Maynard2} already constructed pseudorandom majorants for Chen and bounded gap primes respectively.
We provide here a similar construction.
Let $R=N^\gamma$ for some small $\gamma >0$ to be chosen appropriately.
Relying on Green and Tao's ``smoothed'' approach \cite{GT-linear-equations}, 
we define \begin{equation}
\label{lambdachir}
\Lambda_{\chi,\gamma}(n):=\log R\Big(\sum_{\ell\mid n}\mu (\ell)\chi\left(\frac{\log \ell}{\log R}\right)\Big)^2,
\end{equation}
where
$\chi:\R \rightarrow [0,1]$ is a smooth, even function  supported on $[-2,2]$ satisfying
$\chi(0)=1=\int_0^2\abs{\chi'}^2$.
Finally let
$\Lambda_{\chi,\gamma,\mathcal{H}}(n):=\prod_{h\in \mathcal{H}}\Lambda_{\chi,\gamma}(n+h)$.
Note that this function is periodic (of period $\prod_{\ell\leq R^2}\ell$ for instance), so we extend it on 
$\Z$ as a periodic function.
Once $W$-tricked, this will be a pseudorandom measure. Ultimately, this is a consequence of the following proposition.

\begin{proposition}[Correlations of sieve weights] \label{prop: correlation}
Let $d,t\geq 1$ be integers.
Let $D,\eta>0$.
Let $\Psi=(\psi_1,\ldots,\psi_t)$ be a finite complexity system of affine-linear forms in $d$ variables.
Suppose that $\gamma>0$ is sufficiently small in terms of $d,t$.
Suppose that the linear coefficients, as well as the integers $h_1,\ldots,h_r$, are bounded in magnitude by $D$. Assume that $w$ is sufficiently large in terms of $d,t,D$.
Let $K\subset [-N,N]^d$ satisfy $\Vol(K)\geq \eta N^d$.
Suppose $b_1,\ldots,b_t$ are in $B_{\mathcal{H}}$ (with $B_{\mathcal{H}}$ as in \eqref{equ20}).
Then
$$
\sum_{\mathbf{n}\in K\cap\Z^d}\prod_{i\in [t]}\Lambda_{\chi,\gamma,\mathcal{H}}(W\psi_i(\mathbf{n})+b_i)=\Vol(K)
\left(\frac{W}{\varphi(W)}\right)^{rt}(1+o_{w\rightarrow \infty}(1)+O(e^{\sqrt{w}}/\log^{1/20} N))
$$
where the error terms above may depend on $d,t,D,\eta$ only.
\end{proposition}
\begin{proof}
The left-hand side equals
\begin{align}\label{equ22}
\sum_{\mathbf{n}\in K\cap\Z^d}\prod_{i\in [t],j\in[r]}\Lambda_{\chi,\gamma}(W\psi_i(\mathbf{n})+b_i+h_j).
\end{align}
We then apply \cite[Theorem D.3]{GT-linear-equations} to the system 
$\mathcal{L}=(W\psi_i+b_i+h_j)_{i\in [t],j\in[r]}$, whereby we assume that 
$\gamma$ is small enough in terms of $d,t$.

 Recalling that $\int_{0}^2 |\chi'|^2=1$, \cite[Theorem D.3]{GT-linear-equations} gives for \eqref{equ22} an estimate
\begin{align}\label{equ23}
\Vol(K)  \prod_p \beta_p(\mathcal{L}) +O(N^de^{X}/(\log R)^{1/20}), 
\end{align}
where $X=\sum_{p\in \mathcal{P}}p^{-1/2}$, and $\mathcal{P}$ is the set of \textit{exceptional primes} of $\mathcal{L}$, i.e. those primes $p$ such that modulo $p$ some two of the forms of $\mathcal{L}$ are proportional.
In view of the hypotheses of \cite[Theorem D.3]{GT-linear-equations} (bounded homogeneous coefficients),
one may fear that the implied constant in the big oh term depends 
on the size of the homogeneous coefficients of $\mathcal{L}$, so
ultimately on $w$, but in fact it does not at all as it quickly appears in the proof, since only the behaviour of $\mathcal{L}$ modulo each prime $p$ plays a role in the proof.
This is made clear in \cite[Proposition 2.13]{bienvenuThesis}; in fact already in
\cite[equation (D.24)]{GT-linear-equations} the bound \eqref{equ23} was applied
to a system $\mathcal{L}$ with unbounded coefficients.

Note that the primes in $\mathcal{P}$ are either $\leq w$, or $O_{d,t}(D)$. 
Assuming $w$ is large enough in terms of $d,t,D$, we can assume that they are all $\leq w$,
and therefore $X\ll \sqrt{w}/\log w$, so the error term in \eqref{equ23} becomes $O_{d,t,D,\eta}(\Vol(K)e^{\sqrt{w}}/\log^{1/20}N))$.
Moreover, we have $\beta_p=\beta_p(\mathcal{L})=(p/(p-1))^{rt}$ for $p\leq w$ and $\beta_p=1+O_{d,t,D}(p^{-2})$ as $p$ tends to infinity thanks to Lemma \ref{locfac}, whence $\prod_p\beta_p=(W/\varphi(W))^{rt}(1+o_{w\rightarrow \infty}(1))$. This concludes the proof.
\end{proof}
Now we are ready to prove Proposition \ref{prop: pseudo}.

\begin{proof}[Proof of Proposition \ref{prop: pseudo}]
In this proof, for any $n\in \Z_N$ or any $n\in\Z$,
we will denote by $\tilde{n}$ the unique element of $[N]$ such that $\tilde{n}\equiv n \PMod{N}$.

For $b\in \N$ we define $\nu_b :\Z/N\Z\rightarrow\R_{\geq 0}$
to be the function $n\mapsto\left(\frac{\varphi(W)}{W}\right)^{r}\Lambda_{\chi,\gamma,\mathcal{H}}(W\tilde{n}+b)$ where 
$\gamma\in (0,\rho/2)$. This definition naturally gives rise to a function denoted again by $\nu_b$
on $\Z$.
Further we set $\nu_{\mathbf{b}}(n):=\frac{1}{s}\sum_{i=1}^s\nu_{b_i}(n)$.
Note that
 whenever $N^{2\gamma/\rho}\leq n\leq N$ satisfies $\theta(n)>0$, we have
$\theta(n)\leq\log^r(n+2)\ll \log^rR=\Lambda_{\chi,\gamma,\cH}(n)$.
Hence, whenever $\mathbf{b}=(b_1,\ldots,b_s)$ is in $B_{\mathcal{H}}^s$, we have 
$\theta_{W,b_i}(n)\ll_\gamma \Lambda_{\chi,\gamma,\mathcal{H}}(Wn+b_i)\ll \nu_{\mathbf{b}}(n)$ for each $i\in [s]$ and $n\in [N^{\gamma/\rho},N]$.

Let us verify that $\nu_\mathbf{b}$ is a $(M,\varepsilon)$-pseudorandom measure.
Hence, let $d\leq M$ and $t\leq M$ and $\Psi : \Z^d\rightarrow\Z^t$ be a finite complexity
system of affine-linear forms whose linear coefficients are bounded by $M$.
We work in the regime where $w$ and $N$ tend to infinity while $w\leq \log\log N$.
It suffices to verify that in this regime
\begin{equation}
\label{eq:lincondproof}
\E_{\mathbf{n}\in (\Z/N\Z)^d}\prod_{i\in [t]} \nu_{c_i}(\psi_i(\mathbf{n}))=1+o_M(1)
\end{equation}
for any fixed $\mathbf{c}\in \{b_1,\ldots,b_s\}^t$.
We cannot apply Proposition \ref{prop: correlation} to prove equation \eqref{eq:lincondproof} at this 
stage, since this equation effectively concerns a linear system over $\Z/N\Z$ and not
over $\Z$. In other words, there are wrap-around issues.
To be able to  apply Proposition \ref{prop: correlation}, we first
rewrite the left-hand side of \eqref{eq:lincondproof} as
\begin{equation}
\E_{\mathbf{n}\in (\Z/N\Z)^d}\prod_{i\in [t]} \nu_{c_i}(\psi_i(\mathbf{n}))=\E_{\mathbf{n}\in [N]^d}\prod_{i\in [t]} \nu_{c_i}(\widetilde{\psi_i(\mathbf{n})}).
\end{equation}
Observe that the map $\mathbf{n}\mapsto \widetilde{\psi_i(\mathbf{n})}$ is not an affine-linear map,
so that we still cannot invoke Proposition \ref{prop: correlation}.
However, it is piecewise affine-linear.
To exploit this property, we decompose $[N]^d$ in boxes of the form
$$
B_\mathbf{u}=\left\{\mathbf{x}\in [N]^d\,:\,x_j\in \left(\left\lfloor \frac{(u_j-1)N}{Q}\right\rfloor,\left\lfloor \frac{u_jN}{Q}\right\rfloor\right], j\in [d]\right\}
$$
where $\mathbf{u}$ ranges over $[Q]^d$, and $Q$ is some function of $N$, to be determined later, that tends slowly to infinity
with $N$.
Assuming that $N/Q$ tends to infinity, we have
\begin{equation}
\label{eq:decBoxes}
\E_{\mathbf{n}\in [N]^d}\prod_{i\in [t]} \nu_{c_i}(\widetilde{\psi_i(\mathbf{n})})=
\E_{\mathbf{u}\in [Q]^d}\E_{\mathbf{n}\in B_\mathbf{u}}\prod_{i\in [t]} \nu_{c_i}(\widetilde{\psi_i(\mathbf{n})})
+o(1),
\end{equation}
where the $o(1)$ accounts for the fact that all boxes are not exactly of the same size; they are all of size $(N/Q+O(1))^d=(N/Q)^d(1+o(1))$ though.
Call $\mathbf{u}$ and the corresponding box $B_\mathbf{u}$ \textit{nice} if for every $i\in [t]$ the number $\lceil \psi_i(\mathbf{n})/N\rceil$
is constant as $\mathbf{n}$ ranges in $B_\mathbf{u}$; that is, there exists $k=k_{i,\mathbf{u}}\in\Z$
such that $\psi_i(\mathbf{n})\in (kN,(k+1)N]$ for every $\mathbf{n}\in B_\mathbf{u}$.
When $\mathbf{u}$ is nice, $\widetilde{\psi_i(\mathbf{n})}=\psi_i(\mathbf{n})-k_{i,\mathbf{u}}N\in [N]$ for every $i\in [t]$ and
$\mathbf{n}\in B_\mathbf{u}$. Therefore, 
$$\E_{\mathbf{n}\in B_\mathbf{u}}\prod_{i\in [t]} \nu_{c_i}(\widetilde{\psi_i(\mathbf{n})})=\E_{\mathbf{n}\in B_\mathbf{u}}\prod_{i\in [t]} \nu_{c_i}(\psi_{i,\mathbf{u}}(\mathbf{n}))$$
where the affine-linear map $\Psi_\mathbf{u}:\Z^d\rightarrow\Z^t$ is defined by setting $\psi_{i,\mathbf{u}} : \mathbf{n} \mapsto
\psi_i(\mathbf{n})-k_{i,\mathbf{u}}N$.
It is clear that this map, having the same homogeneous part as $\Psi$, is still of finite complexity
and has bounded homogeneous coefficients.

Thus, we may now apply Proposition \ref{prop: correlation}
to conclude that 
$$\E_{\mathbf{n}\in B_\mathbf{u}}\prod_{i\in [t]} \nu_{c_i}(\psi_{i,\mathbf{u}}(\mathbf{n}))=1+o_{w\rightarrow\infty;M}(1)+o_{N\rightarrow\infty;M}(1).$$
It remains to handle the other boxes.
Suppose that $\mathbf{u}$ is not nice. Thus, there exists $i\in [t]$
and two vectors $\mathbf{x},\mathbf{y}$ in $B_\mathbf{u}$
 such that 
$k:=\lceil \psi_i(\mathbf{x})/N\rceil< \lceil \psi_i(\mathbf{y})/N\rceil$. However,
we have $\abs{\psi_i(\mathbf{x})- \psi_i(\mathbf{y})}\leq 2dM(N/Q+1)<N$, if $Q>3dM$.
Therefore, 
$$
\psi_i(\mathbf{x})/N\leq k<\psi_i(\mathbf{y})/N\leq \psi_i(\mathbf{x})/N+O_M(1/Q)
$$
and
$$
\psi_i(\mathbf{x})/N\geq \psi_i(\mathbf{y})/N-O_M(1/Q)\geq k-
O_M(1/Q).
$$
The last two displayed lines show that both $\psi_i(\mathbf{x})$ and $ \psi_i(\mathbf{y})$ are $kN+O(N/Q)$;
thus $\psi_i(\mathbf{n})=O_{M}(N/Q)\PMod{N}$ for all $\mathbf{n}\in B_\mathbf{u}$.
Further, there exists an integer $k=k_{i,\mathbf{u}}$ such that for all $\mathbf{n}\in B_\mathbf{u}$
and $i\in [t]$
either $\psi_i(\mathbf{n})-kN\in [N]$ or $\psi_i(\mathbf{n})-(k+1)N\in [N]$; consequently,
\begin{align*}
\nu_{c_i}(\widetilde{\psi_i(\mathbf{n})})&=\nu_{c_i}(\psi_i(\mathbf{n})-k_{i,\mathbf{u}}N)1_{\psi_i(\mathbf{n})-k_{i,\mathbf{u}}N\in [N]}+\nu_{c_i}(\psi_i(\mathbf{n})-(k_{i,\mathbf{u}}+1)N)1_{\psi_i(\mathbf{n})-(k_{i,\mathbf{u}}+1)N\in [N]}\\
&\leq \nu_{c_i}(\psi_i(\mathbf{n})-k_{i,\mathbf{u}}N)+\nu_{c_i}(\psi_i(\mathbf{n})-(k_{i,\mathbf{u}}+1)N).
\end{align*}
Whence the bound 
\begin{multline}
\label{eq:2indicatrices}
\E_{\mathbf{n}\in B_\mathbf{u}}\prod_{i\in [t]} \nu_{c_i}(\widetilde{\psi_i(\mathbf{n})})
\leq \E_{\mathbf{n}\in B_\mathbf{u}}\prod_{i\in [t]}(\nu_{c_i}(\psi_i(\mathbf{n})-k_{i,\mathbf{u}}N)+\nu_{c_i}(\psi_i(\mathbf{n})-(k_{i,\mathbf{u}}+1)N)).
\end{multline}
Expanding the product makes the right-hand side of inequality \eqref{eq:2indicatrices} the sum of $2^t$ averages, each of which equals $1+o(1)$ by Proposition \ref{prop: correlation}.
So the left-hand side of inequality \eqref{eq:2indicatrices} is $O(1)$.
It remains to prove that non-nice boxes are rare.
Suppose that $\mathbf{u}$ is not nice. As pointed out above, if $Q$ is large enough, 
there exists $i\in [t]$ such that
$\psi_i(\mathbf{n})=O_{M}(N/Q)\PMod{N}$ for all $\mathbf{n}\in B_\mathbf{u}$.
On the other hand,
\begin{align*}
\psi_i(\mathbf{n})&=\psi_i(\lfloor N\mathbf{u}/Q\rfloor)+O_{M}(N/Q)\\
&=\psi_i(N\mathbf{u}/Q+O(1))+O_{M}(N/Q)\\
&=N\dot{\psi_i}(\mathbf{u})/Q+\psi_i(0)+O_{M}(N/Q).
\end{align*}
Dividing by $N/Q$ yields 
\begin{equation}
\label{eq:fewsols}
\dot{\psi_i}(\mathbf{u})+Q\psi_i(0)/N=O(1)\PMod{Q}.
\end{equation}
Now $\dot{\psi_i}\neq 0$ and when $Q$ is large enough $\dot{\psi_i}\neq 0\PMod{Q}$ as well, so the number of solutions $\mathbf{u}\in [Q]^d$ to the estimate 
\eqref{eq:fewsols} is $O(Q^{d-1})$.
Multiplying by $t$, the proportion of bad boxes among all boxes is therefore $O(Q^{-1})=o(1)$, the implied constant
depending on $M$ only.
This concludes the proof of the estimate \eqref{eq:decBoxes}.
\end{proof}

\section{Almost twin primes in generalized Bohr sets}\label{sec:twin-prime-bohr}
Now we want to prove hypothesis (ii) in Theorem~\ref{th:transference} for the $W$-tricked functions $\theta_1$ and $\theta_2$.
Thus we need to find almost twin primes in $s$-measurable sets.
We start with bounded gap primes. The next proposition establishes hypothesis (ii) in Theorem~\ref{th:transference} for a function of the form
$\theta_{W,b}$ from Section \ref{sec: w}, upon letting $b_i=b+h_i$.

\begin{proposition}[Bounded gap primes with nilsequences]\label{prop:twin-prime-nil}

Fix positive integers $m, d, \Delta$, and some $\varepsilon>0$, $K \geq 2$. Also let $w\geq 1$ be sufficiently large in terms of $m,d,\Delta,\ee, K$ and let $W=\prod_{p\leq w}p$. There exist $\rho = \rho(m) > 0$,  and a positive integer $k = k(m)$, such that the following statement holds for sufficiently large $x \geq x_0(m,d,\Delta,\varepsilon,K,w)$.

Let $\xi \in \Xi_d(\Delta,K)$ be a nilsequence taking values in $[0,1]$.
Let $b_1,\cdots,b_k$ be distinct integers satisfying $(b_i,W)=1$ and $|b_i|\leq \log x$ for each $i\in[k]$. Then
\begin{align*}
\sum_{\substack{n \leq x \\ |\{Wn+b_1, \ldots, Wn+b_k\} \cap \Prime| \geq m \\ p \mid \prod_{i=1}^k (Wn+b_i) \implies p > x^{\rho}}} \xi(n) \gg_m  \left(\prod_p \beta_p\right) \frac{1}{(\log x)^k} \left( \sum_{n \leq x} \xi(n) - \varepsilon x \right), 
\end{align*}
where  $\beta_p = \beta_p(\CL)$ is the local factor defined as in Definition~\ref{def:locFac} for the system of affine-linear forms $\CL = \{L_1,\ldots,L_k\}$ with $L_i(n) = Wn+b_i$.
\end{proposition}

We remark that if $p \leq w$ then $\beta_p = (p/\varphi(p))^k$,  and if $p > w$ then, writing $a_p$ for the number of distinct residue classes among $b_1,\ldots,b_k\PMod{p}$,
\begin{equation}\label{eq:betap} 
\beta_p = \left(\frac{p}{\varphi(p)}\right)^k \left(1 - \frac{a_p}{p}\right) = \left(\frac{p}{\varphi(p)}\right)^{k-a_p}\left(1 + O_k(p^{-2})\right).
\end{equation}
In particular, if $w>\abs{b_i-b_j}$ for all $i,j$, we infer that $\prod_p\beta_p\gg (W/\phi(W))^k$.

We now turn to the corresponding statement for Chen primes.

\begin{proposition}[Chen primes with nilsequences]\label{prop:chen-prime-nil}
Fix positive integers $d, \Delta$ and some  $\varepsilon>0$, $K \geq 2$. Also let $w\geq 1$ be sufficiently large in terms of $d,\Delta,\ee, K$ and $W=\prod_{p\leq w}p$. The following statement holds for sufficiently large $x \geq x_0(d,\Delta,\varepsilon,K,w)$.

Let $\xi \in \Xi_d(\Delta,K)$ be a nilsequence taking values in $[0,1]$.
Then for some absolute constant $\delta_0>0$ and any $1\leq b\leq W$ with $(b,W)=(b+2,W)=1$ we have
\[ \sum_{\substack{n\leq x \\ Wn+b\in  \Prime\\ Wn+b+2\in P_2\\p \mid Wn+b+2\implies p\geq x^{1/10}}} \xi(n) \geq \left(\frac{W}{\varphi(W)}\right)^2 \frac{\delta_0}{(\log x)^2} \left( \sum_{n\leq x} \xi(n) - \varepsilon x \right). \]
\end{proposition}

We will prove Propositions~\ref{prop:twin-prime-nil} and~\ref{prop:chen-prime-nil} by reducing to the case when the underlying polynomial sequence is equidistributed, Propositions~\ref{prop:twin-prime-eq-nil} and~\ref{prop:chen-prime-eq-nil}, using a factorization theorem for nilsequences.

\begin{proposition}[Bounded gap primes weighted by equidistributed nilsequences]\label{prop:twin-prime-eq-nil}
Fix positive integers $m, d, \Delta$, and some $\varepsilon>0$, $A \geq 2$. There exist $\rho = \rho(m) > 0$,  a positive integer $k = k(m)$, and $C = C(m,d,\Delta) > 0$, such that the following statement holds for sufficiently large $x \geq x_0(m,d,\Delta,\varepsilon, A)$.

Let $K \geq 2$ and $\eta \in (0,1/2)$ be parameters satisfying the conditions
\[ \eta \leq K^{-C} (\log x)^{-CA}, \ \ K \leq (\log x)^C. \]
Let $\xi \in \Xi_d(\Delta, K; \eta, x)$ be a nilsequence taking values in $[0,1]$. Let $\CL = \{L_1,\ldots, L_k\}$ be an admissible $k$-tuple of linear functions with $L_i(n)=a_in+b_i$ and $1\leq a_i \leq (\log x)^A$, $|b_i|\leq x$. Then
\begin{align}\label{equ7} \sum_{\substack{n\leq x \\ |\{L_1(n), \ldots, L_k(n)\} \cap \Prime| \geq m \\ p \mid \prod_{i=1}^k L_i(n) \implies p > x^{\rho}}} \xi(n) \gg_m  \frac{\FS(\mathcal{L})}{(\log x)^k} \left( \sum_{n\leq x} \xi(n) - \varepsilon x \right), \end{align}
where the singular series is given by 
$\mathfrak{S}(\mathcal{L}):=\prod_{p}\beta_p(\mathcal{L})$, i.e.
\begin{align}\label{eq27}
\mathfrak{S}(\mathcal{L})=\prod_{p}\Big(1-\frac{1}{p}\Big)^{-k}\Big(1-\frac{|\{n\in \mathbb{Z}/p\mathbb{Z}:  L_1(n)\cdots L_k(n)\equiv 0\PMod p\}|}{p}\Big)>0.   
\end{align}
\end{proposition}

\begin{proposition}[Chen primes weighted by equidistributed nilsequences]\label{prop:chen-prime-eq-nil}
Fix positive integers $d, \Delta$ and some  $\varepsilon>0$, $A \geq 2$. There exists $C = C(d,\Delta) > 0$, such that the following statement holds for sufficiently large $x \geq x_0(d,\Delta,\varepsilon,A)$.

Let $K \geq 2$ and $\eta \in (0,1/2)$ be parameters satisfying the conditions
\[ \eta \leq K^{-C} (\log x)^{-CA}, \ \ K \leq (\log x)^C. \]
Let $\xi \in \Xi_d(\Delta, K; \eta, x)$ be a nilsequence taking values in $[0,1]$.  Let $\CL = \{L_1,L_2\}$ be an admissible set of two linear functions with $L_1(n) = an+b$ and $L_2(n) = an+b+2$, where $1 \leq a\leq \log x$, $|b|\leq x$. Then for some absolute constant $\delta_0>0$ we have
\[ \sum_{\substack{n\leq x \\ L_1(n)\in  \Prime\\ L_2(n)\in P_2\\p \mid L_2(n)\implies p\geq x^{1/10}}} \xi(n) \geq  \delta_0 \frac{\FS(\CL)}{(\log x)^2} \left( \sum_{n\leq x} \xi(n) - \varepsilon x \right), \]
where the singular series is given by \eqref{eq27}.
\end{proposition}

The purpose of this section is to deduce Propositions~\ref{prop:twin-prime-nil} and~\ref{prop:chen-prime-nil} from the equidistributed case, Propositions~\ref{prop:twin-prime-eq-nil} and~\ref{prop:chen-prime-eq-nil}. We will collect some sieve lemmas in Section~\ref{sec:sieve} and some analytic inputs of Bombieri--Vinogradov type in Section~\ref{sec:BV} before proving Propositions~\ref{prop:twin-prime-eq-nil} and~\ref{prop:chen-prime-eq-nil} in Section~\ref{sec:eq}.

\subsection{Dealing with the periodic case}

In the deduction process, we need to deal with a local (modulo $q$) version of Propositions~\ref{prop:twin-prime-nil}
and \ref{prop:chen-prime-nil}, where the requirement that all of the $Wn+b_i$'s are almost primes is replaced by the local conditions that $(Wn+b_i,q)=1$.

\begin{lemma}\label{lem:twim-prime-periodic}
Fix positive integers $k, d, \Delta$, and some $\varepsilon>0$, $K \geq 2$. Also let $w\geq 1$ be sufficiently large in terms of $k,d,\Delta,\ee, K$ and $W=\prod_{p\leq w}p$. Then the following statement holds for sufficiently large $x \geq x_0(k,d,\Delta,\varepsilon,K,w)$.

Let $\xi \in \Xi_d(\Delta,K)$.
Let $\{b_1,\ldots,b_k\}$ satisfy $(b_i,W)=1$ for every $i\in [k]$. Let $q \leq x^{0.9}$ be a positive integer with $(q,W)=1$. Then
\begin{align*}
\Big|\beta^{-1} \left(\frac{q}{\varphi(q)}\right)^k \sum_{\substack{n \leq x \\ (\prod_{i=1}^k(Wn+b_i),q)=1}} \xi(n) - \sum_{n \leq x} \xi(n)\Big| \leq  \varepsilon x,
\end{align*}
where $\beta = \prod_{p \mid q}\beta_p$, and $\beta_p$ is defined as in Proposition~\ref{prop:twin-prime-nil}.
\end{lemma}

\begin{proof}
Let $X$ be the set of $n \leq x$ such that $(Wn+b_i,q)=1$ for each $1 \leq i \leq k$. Consider the function
\[ f(n) = \beta^{-1} \left(\frac{q}{\varphi(q)}\right)^k 1_X(n) - 1. \]
Let us prove first that
\[ \|f\|_{U^{d+1}[x]} = o_{x\rightarrow\infty;k,d}(1) + o_{w\rightarrow\infty; k,d}(1). \]
Expanding out $\|f\|_{U^{d+1}[x]}$ and letting $g : \Z^{d+2}\rightarrow \Z^{\{0,1\}^{d+1}}$ be the Gowers norm system $(x,\mathbf{h})\mapsto (x+\boldsymbol\omega
\cdot \mathbf{h})_{\boldsymbol\omega
\in\{0,1\}^{d+1}}$, we are left with the task of proving that
\begin{equation}
\label{eq:expansion}
\sum_{\mathbf{n}\in \D\cap\Z^{d+2}}\prod_{\omega\in \{0,1\}^{d+1}}f(g_{\boldsymbol\omega}(\mathbf{n}))=o_{x\rightarrow\infty;k,d}(x^{d+2}) + o_{w\rightarrow\infty; k,d}(x^{d+2}),
\end{equation}
where $\D=\{\mathbf{y}\in\R^{d+2} : g_{\boldsymbol\omega}(\mathbf{y})\in [1,x],\  \forall\boldsymbol\omega
\in\{0,1\}^{d+1} \}$.
Expanding further, the left-hand side of \eqref{eq:expansion} equals
\begin{equation}
\label{eq:expansion2}
\sum_{\Omega\subset \{0,1\}^{d+1}}(-1)^{\abs{\Omega}}\sum_{\mathbf{n}\in \D\cap\Z^{d+2}}\prod_{\boldsymbol\omega\in \Omega} \beta^{-1} \left(\frac{q}{\varphi(q)}\right)^k 1_X(g_{\boldsymbol\omega}(\mathbf{n})) = \sum_{\Omega\subset \{0,1\}^{d+1}} (-1)^{|\Omega|} S_{\Omega},
\end{equation}
where, after a change of variables, we have
\begin{equation}
\label{eq:decomposed}
S_{\Omega} = \sum_{\mathbf{a}\in [q]^{d+2}}\sum_{\substack{\mathbf{n}\in\Z^{d+2}\\q\mathbf{n}+\mathbf{a}\in \D}}
\prod_{\boldsymbol\omega\in \Omega} \beta^{-1} \left(\frac{q}{\varphi(q)}\right)^k 1_X(g_{\boldsymbol\omega}(q\mathbf{n}+\mathbf{a})).
\end{equation}
Now the summand of the inner sum actually does not depend on $\mathbf{n}$ since $1_X(g_{\boldsymbol\omega}(q\mathbf{n}+\mathbf{a}))=1_X(g_{\boldsymbol\omega}(\mathbf{a}))$.
Let $\D_\mathbf{a}=\{\mathbf{n}\in\R^{d+2}:q\mathbf{n}+\mathbf{a}\in \D\}$, which is a convex
body of volume $q^{-(d+2)}\Vol(\D)$. Since $\D_{\ve{a}} \subset [1, x/q]^{d+2}$ and $\Vol(\D) \gg_d x^{d+2}$, the number of integral points $\ve{n} \in \D_{\ve{a}} \cap \Z^{d+2}$ is
\[ \Vol(\D_{\ve{a}}) + O_d\left( (x/q)^{d+1} \right) = q^{-(d+2)} \Vol(\D) \left(1 +  O_d(q/x) \right). \]
It follows that
\[ S_{\Omega} = \left(1 +  O_d(q/x) \right) \Vol(\D) \cdot \beta^{-|\Omega|}   \E_{\ve{a} \in [q]^{d+2}} \prod_{\boldsymbol\omega\in \Omega} \left(\frac{q}{\varphi(q)}\right)^{k} 1_X(g_{\boldsymbol\omega}(\mathbf{a}))  .\]
By multiplicativity and the definition of the singular series (Definition~\ref{def:locFac}), the average over $\ve{a}$ above can be written as
\[
\prod_{p\mid q}  \left(\E_{\ve{a} \in (\Z/p\Z)^{d+2}} \prod_{\boldsymbol\omega\in \Omega} \prod_{i=1}^k \frac{p}{\varphi(p)} 1_{(Wg_{\ve{\omega}}(\ve{a}) + b_i, p)=1}  \right) = \prod_{p \mid q}  \beta_p(\mathcal{G}_{\Omega}),
\]
where $\mathcal{G}_{\Omega}$  is the system of affine-linear forms $\mathcal{G}_{\Omega}$ consisting of $\ve{a} \mapsto W g_{\ve{\omega}}(\ve{a}) + b_i$ for $\ve{\omega} \in \Omega$ and $1 \leq i \leq k$. If there are $a_p$ distinct residue classes among $b_1,\cdots,b_k\PMod{p}$, then $\mathcal{G}_{\Omega}$ consists of $a_p|\Omega|$ distinct affine-linear forms modulo $p$, no two of which are linearly dependent (over $\F_p$). Hence, Lemma~\ref{locfac} implies that
\[ \beta_p(\mathcal{G}_{\Omega}) = \left(\frac{p}{\varphi(p)}\right)^{(k-a_p)|\Omega|} \left(1 + O_{k,d}(p^{-2})\right). \]
Since $p\mid q\Rightarrow p>w$, we have $\prod_{p \mid q} (1+ O_{k,d}(p^{-2})) = 1 + O_{k,d}(w^{-1})$. Putting things together, we have
\[ S_{\Omega} = \left(1 +  O_d(q/x) + O_{k,d}(w^{-1}) \right) \Vol(\D) \cdot \beta^{-|\Omega|} \prod_{p \mid q} \left(\frac{p}{\varphi(p)}\right)^{(k-a_p)|\Omega|}.  \]
From~\eqref{eq:betap} we deduce that
\[ S_{\Omega} = \left(1 + O_{k,d}(qx^{-1} + w^{-1})\right) \Vol(\D) \]
for each $\Omega \subset \{0,1\}^{d+1}$, and this establishes~\eqref{eq:expansion}.

By~\cite[Proposition 11.2]{GT-linear-equations} (see also footnote \ref{foot1}), the nilsequence $\xi$ can be decomposed as
\[ \xi = \xi_1 + \xi_2, \]
where $\|\xi_1\|_{U^{d+1}[x]^*} = O_{d,\Delta,K,\ee}(1)$ and $\|\xi_2\|_{\infty} \leq \ee/4$. Hence,
\[ \Big|\sum_{n \leq x} f(n) \xi_1(n) \Big| \leq x \|f\|_{U^{d+1}[x]} \cdot \|\xi_1\|_{U^{d+1}[x]^*} \leq \tfrac{\ee}{2} x, \]
provided that $w$ and $x$ are large enough in terms of $k,d,\Delta,K,\ee$, and
\[ \Big|\sum_{n \leq x} f(n) \xi_2(n) \Big| \leq \|\xi_2\|_{\infty} \sum_{n \leq x} |f(n)|  \leq \tfrac{\ee}{2} x. \]
Combining the two inequalities above gives
\[ \Big|\sum_{n \leq x} f(n) \xi(n) \Big| \leq \ee x, \]
as desired.
\end{proof}

\subsection{Reducing to the equidistributed case}
We now complete the proof of Propositions \ref{prop:twin-prime-nil} and \ref{prop:chen-prime-nil} assuming Propositions~\ref{prop:twin-prime-eq-nil} and~\ref{prop:chen-prime-eq-nil}.
Let $f : \N\rightarrow\C$ be any function satisfying $\abs{f(n)}\leq 1$ for any $n\in \N$; we will specialize later to the case where $f$ is
the indicator functions of the sets over which the summations in these propositions run.
Let $\xi : \N\rightarrow [0,1]$ be a nilsequence in $\Xi_d(\Delta,K)$.
We want to estimate $\sum_{n\leq x} f(n)\xi(n)$.
By Definition \ref{def:PsiDeltaK}, there exists
 a nilmanifold $G/\Gamma$ of dimension at most $\Delta$, equipped with a filtration $G_{\bullet}$  of degree $\leq d$ and a
$K$-rational  Mal'cev basis  $\mathcal{X}$, a
  polynomial sequence $g: \Z \to G$ adapted to  $G_{\bullet}$ and a
 Lipschitz function $F: G/\Gamma \to  \mathbb{C}$  satisfying $\|F\|_{\operatorname{Lip}(\mathcal{X})} \leq 1$, such that $\xi(n)=F(g(n)\Gamma)$.

Let 
\begin{align*}
\mu:=\frac{1}{x}\sum_{n\leq x}\xi(n).    
\end{align*}
We may assume that $\mu\geq \varepsilon$, as otherwise there is nothing to prove.

Let $B = B(m,d,\Delta) > 0$ be sufficiently large. To reduce Propositions \ref{prop:twin-prime-nil} and \ref{prop:chen-prime-nil} to the case when $g$ is equidistributed, we apply the factorization theorem \cite[Theorem 1.19]{GT-nilsequence} to obtain some parameter $M \in [\log x, (\log x)^{O_{B,d,\Delta}(1)}]$ and a decomposition $g = \epsilon g' \gamma$ into polynomial sequences $\epsilon, g', \gamma: \Z \to G$ with the following properties:
\begin{enumerate}
\item $\epsilon$ is $(M,x)$-smooth, i.e. $d(\epsilon(n), \textnormal{id}_G)\leq  M$ and $d(\epsilon(n), \epsilon(n-1)) \leq     M/x$ for all $n\in [x]$, with $d=d_{\mathcal{X}}$ the metric used on $G$. 

\item $g'$ takes values in a rational subgroup $G' \subseteq G$, equipped with a Mal'cev basis $\CX'$ in which each element is an $M$-rational combination of the elements of $\CX$, and moreover $\{g'(n)\}_{n \leq x}$ is totally $M^{-B}$ equidistributed in $G'/\Gamma \cap G'$.
\item $\gamma$ is $M$-rational (so that $\gamma(n)\Gamma$ is an $M$-rational point for every $n \in \Z$), and moreover $\{\gamma(n)\Gamma\}_{n \in \Z}$ is periodic with period some $q \leq M$.
\end{enumerate}

In the case of Proposition~\ref{prop:twin-prime-nil}, we may make the following additional assumption on $q$ by enlarging $q$ and $M$ if necessary: If $b_i \equiv b_j\PMod{p}$ for some $i \neq j$ and some prime $p$, then $p$ divides $q$. By~\eqref{eq:betap}, this implies that if $p\nmid qW$, then $\beta_p = 1 + O_k(p^{-2})$.

Let $\mathcal{Q}$ be a collection of arithmetic  progressions of step $q$ and length $\asymp (x/qM) (\log x)^{-100}$ such that $[x]=\bigcup_{P\in \mathcal{Q}}P$.
 For each $P\in \mathcal{Q}$, let $\gamma_P$ be the (constant) value of $\gamma$ on $P$ and consider
\begin{equation}\label{equ19} 
\sum_{n \in P} f(n)\xi(n)=\sum_{n \in P} f(n)F(\epsilon(n) g'(n) \gamma_P \Gamma).
\end{equation}
We shall first dispose of the smooth part $\epsilon(n)$. Pick an arbitrary $n_P \in P$, and let $\epsilon_P = \epsilon(n_P)$. 
If $n\in P$, then $|n-n_P| \ll (x/M)(\log x)^{-100}$. Since the Lipschitz norm of $F$ is bounded by $1$, we have 
\begin{align*}
|F(\epsilon(n)g'(n)\gamma_P\Gamma)-F(\epsilon_P g'(n)\gamma_P\Gamma)|&\leq  d_{\mathcal{X}}(\epsilon(n)g'(n)\gamma_P,\epsilon(n_P)g'(n)\gamma_P)\\
&= d_{\mathcal{X}}(\epsilon(n),\epsilon(n_P))\\
&\leq \frac{M}{x} |n-n_P| \ll (\log x)^{-100}
\end{align*}
where we used the right-invariance of the metric $d$.
Hence~\eqref{equ19} equals
\[  \sum_{n \in P} f(n)F(\epsilon_P g'(n) \gamma_P \Gamma) + O(|P| (\log x)^{-100}). \]
Let $H_P$ be the conjugate $H_P = \gamma_P^{-1} G' \gamma_P$, let $\Gamma_P = H_P \cap \Gamma$, let $F_P: H_P \to [0,1]$ be the $\Gamma_P$-automorphic function defined by $F_P(x) = F(\epsilon_P \gamma_P x)$, and let $g_P: \Z \to H_P$ be the polynomial sequence defined by $g_P(n) =  \gamma_P^{-1}g'(n)\gamma_P$. Thus
\[ F(\epsilon_P g'(n)\gamma_P\Gamma) = F_P( g_P(n)\Gamma_P). \]
Some routine arguments (see the Claim at the end of Section 2 in \cite{GT-mobius-nil}) produce the following properties:
\begin{enumerate}
\item The sub-nilmanifold $H_P/\Gamma_P$ is equipped with a Mal'cev basis $\CX_P$ in which each element is an $M^{O_{d,\Delta}(1)}$-rational combination of the elements of $\CX$.
\item $\{g_P(n)\}_{n \leq x}$ is totally $M^{-cB}$-equidistributed for some constant $c = c(d, \Delta) > 0$. By choosing $B$ large enough we may ensure that $cB \geq C$, the constant from Proposition~\ref{prop:twin-prime-eq-nil} or Proposition~\ref{prop:chen-prime-eq-nil}.
\item $\|F_P\|_{\operatorname{Lip}} \leq M^{O_{d,\Delta}(1)}$.
\end{enumerate}
Let $y = |P|\gg (x/qM) (\log x)^{-100} \geq x^{1/2}$, and write $P = \{qn + t: n\leq y\}$ for some $t \in \Z$. Let $g_{P}': \Z \to H_P$ be the polynomial sequence defined by $g_{P}'(n) = g_{P}(qn+t)$, so that $\{g_P'(n)\}_{n \leq y}$ is still totally $M^{-cB}$-equidistributed (after possibly reducing the constant $c$). 
Then 
\begin{equation}
\label{eq:sumf}
 \sum_{n \in P} f(n)F(\epsilon_P g'(n) \gamma_P \Gamma)=\sum_{n\leq y}f(qn+t)F_P(g_{P}'(n)\Gamma_P).
\end{equation} 

\textbf{Case of Proposition \ref{prop:twin-prime-nil}.} 
Now we specialize to the function $f$ relevant for Proposition \ref{prop:twin-prime-nil}.
Write $L_i(n) = Wn+b_i$, and let $\CL' = \CL'_P= \{L_1', \cdots, L_k'\}$, where $L_i'$ is the linear function defined by $L_i'(n) = L_i(qn+t)$. 
Let $f$ be the indicator function of the set of integers $n$ such that $\#(\{L_1(n), \ldots, L_k(n)\} \cap \Prime) \geq m$ and $p \mid L_1(n)\cdots L_k(n) \implies p \geq x^{\rho}$.
Thus 
\[\sum_{n\leq y}f(qn+t)F_P(g_{P}'(n)\Gamma_P)
= \sum_{\substack{n\leq y \\ \#(\{L_1'(n), \ldots, L_k'(n)\} \cap \Prime) \geq m \\ p \mid L_1'(n)\cdots L_k'(n) \implies p \geq x^{\rho}}} F_P(g_{P}'(n)\Gamma_P).
\]
If $\CL'$ remains admissible, then by Proposition~\ref{prop:twin-prime-eq-nil} the right-hand side above is
\[
\begin{split}
&\gg_k \frac{\FS(\CL')}{(\log x)^k} \left(\sum_{n\leq y} F_P(g_{P}'(n)\Gamma_P) - \varepsilon/4\cdot y\right)
= \frac{\FS(\CL')}{(\log x)^k} \left(\sum_{n\in P} F(\epsilon_P g'(n)\gamma_P\Gamma) - \varepsilon/4\cdot |P|\right) \\
&\geq   \frac{\FS(\CL')}{(\log x)^k} \left(\sum_{n\in P} F(\epsilon(n)g'(n)\gamma_P\Gamma) - \varepsilon/2\cdot |P|\right),
\end{split}
\]
where the last inequality follows once again from the smoothness of $\epsilon$.
By the definition of $\FS(\CL')$, we see that $\CL'$ is admissible if any only if $(Wt+b_i, q)=1$ for each $1 \leq i \leq k$, and in this case we have
\[ \FS(\CL') = \prod_{p \mid qW} \left(1 - \frac{1}{p}\right)^{-k} \prod_{p \nmid qW} \left(1 - \frac{k}{p}\right) \left(1 - \frac{1}{p}\right)^{-k} \asymp \left(\frac{qW}{\varphi(qW)}\right)^k. \]
Putting everything together, 
under the assumption that $\CL'$ is admissible,
we have proven that
\[ \sum_{\substack{n \in P \\ |\{Wn+b_1, \ldots, Wn+b_k\} \cap \Prime| \geq m \\ p \mid \prod_{i=1}^k (Wn+b_i) \implies p > x^{\rho}}} \xi(n)\gg_m \left(\frac{qW}{\varphi(qW)}\right)^k \frac{1}{(\log x)^k}  \sum_{n \in P} \left(\xi(n) - \tfrac{\ee}{2}\right).  \]
Summing this estimate over all $P\in\mathcal{Q}$, we get
\[
\begin{split}
\sum_{\substack{n \leq x \\ |\{Wn+b_1, \ldots, Wn+b_k\} \cap \Prime| \geq m \\ p \mid \prod_{i=1}^k (Wn+b_i) \implies p > x^{\rho}}} \xi(n) &\gg_m \left(\frac{qW}{\varphi(qW)}\right)^k \frac{1}{(\log x)^k}  \sum_{\substack{n \leq x \\ (\prod_{i=1}^k (Wn+b_i),q)=1}} \left(\xi(n) - \tfrac{\ee}{2}\right) \\
&= \left(\frac{W}{\varphi(W)(\log x)}\cdot \frac{\frac{q}{(q,W)}}{\varphi(\frac{q}{(q,W)})}\right)^k \sum_{\substack{n \leq x \\ (\prod_{i=1}^k (Wn+b_i),\frac{q}{(q,W)})=1}} \left(\xi(n) - \tfrac{\ee}{2}\right).
\end{split}
\]
Finally, applying Lemma~\ref{lem:twim-prime-periodic} to the summation on the right-hand side, with $q$ replaced by $q/(q,W)$) and $\xi$ replaced by $\xi - \ee/2$, we get that the right-hand side above is
at least
\[  \left(\frac{W}{\varphi(W) (\log x)}\right)^k \left(\prod_{p  \mid q/(q,W)=1} \beta_p\right) \left(\sum_{n \leq x} \xi(n) - \ee x\right) = \left(\prod_{p \mid qW} \beta_p\right) \frac{1}{(\log x)^k} \left(\sum_{n \leq x} \xi(n) - \ee x\right).
\]
The conclusion of Proposition \ref{prop:twin-prime-nil} follows, since $\beta_p = 1 + O_k(p^{-2})$ for $p\nmid qW$ by our assumption on $q$.

\textbf{Case of Proposition \ref{prop:chen-prime-nil}.} 
Finally, we address Proposition \ref{prop:chen-prime-nil}. Thus we return to equation \eqref{eq:sumf}
and now specialize to the case where $f$ is the indicator function of the set of integers $n$
such that $L_1(n)\in  \Prime$ and  $L_2(n)\in P_2$ and $p \mid L_2(n)\implies p\geq x^{1/10}$
where $L_1(n)=Wn+b$ and $L_2(n)=Wn+b+2$.
Let $\CL'=\CL'_P=\{L'_1,L'_2\}$ where
for $i\in \{1,2\}$, the linear function $L_i'$ is defined by $L_i'(n) = L_i(qn+t)$.
Then we have 
\[\sum_{n\leq y}f(qn+t)F_P(g_{P}'(n)\Gamma_P)= \sum_{\substack{n\leq y \\ L'_1(n)\in  \Prime\\ L'_2(n)\in P_2\\p \mid L'_2(n)\implies p\geq x^{1/10}}} F_P(g_{P}'(n)\Gamma_P).\]
If $\CL'$ remains admissible (which happens precisely when $(Wt+b,q)=(Wt+b+2,q)=1$), then Proposition~\ref{prop:chen-prime-eq-nil} 
and the same argument as above prove that
the right-hand side above is at least
$$
\delta_0\left(\frac{qW}{\varphi(qW)}\right)^2 \frac{1}{(\log x)^2}  \sum_{n \in P} (\xi(n) - \ee/2).
$$
This time the summation over all
$P\in \mathcal{Q}$ yields, after applying Lemma~\ref{lem:twim-prime-periodic},
\[
\sum_{\substack{n\leq x \\ L_1(n)\in  \Prime\\ L_2(n)\in P_2\\p \mid L_2(n)\implies p\geq x^{1/10}}}\xi(n)
\geq \delta_0 \left(\prod_{p \mid q/(q,W)} \beta_p\right) \left(\frac{W}{\varphi(W)}\right)^2 \frac{1}{(\log x)^2} \left(\sum_{n \leq x} \xi(n) - \ee x\right).
\]
The conclusion of Proposition \ref{prop:chen-prime-nil} follows, since $\beta_p = 1 + O_k(p^{-2})$ for $p > w$ by~\eqref{eq:betap}.

\section{Three sieve lemmas}\label{sec:sieve}

In the proof of Theorems \ref{th:linEqTwinPrimes} and \ref{th: Linnik}, we will need weighted versions of Maynard's sieve for bounded gap primes, Chen's sieve for almost twin primes, and Iwaniec's sieve for primes of the form $x^2+y^2+1$.


\begin{proposition}[Maynard's sieve] \label{prop: maynard}  For any $\theta\in (0,1)$, $k \in \mathbb{N}$, there exist constants $C=C(\theta)$, $\rho=\rho(\theta,k)$ such that the following holds.

Let $(\omega_n)_{n\leq x} $ be any nonnegative sequence, and let $\mathcal{L}=(L_1,\ldots, L_k)$ be an admissible $k$-tuple of linear fuctions with $L_i(n)=a_in+b_i$  and $1\leq a_i,b_i\leq  x$. Suppose that $(\omega_n)$ obeys  the  following hypotheses.
\begin{enumerate}[label=\upshape(\roman*)]
    \item (Prime number theorem) For each $1\leq i\leq k$ and some $\delta > 0$, we have 
\begin{align*}
\frac{\varphi(a_i)}{a_i} \sum_{\substack{n\leq x\\L_i(n)\in \mathbb{P}}}\omega_n \geq \frac{\delta}{\log x}\sum_{n\leq x}\omega_n.
\end{align*}

\item   (Well-distribution in arithmetic progressions)   For some $C_0>0$ we have
\begin{align*}
\sum_{r\leq x^{\theta}}\,\,\max_{c\PMod r}\Big|\sum_{\substack{n\leq x\\n\equiv c\PMod r}}\omega_n-\frac{1}{r}\sum_{n\leq x}\omega_n\Big|\leq C_0 \frac{\sum_{n\leq x}\omega_n}{(\log x)^{101k^2}}.    
\end{align*}

\item  (Bombieri--Vinogradov) For each $1\leq i\leq k$  we have
\begin{align*}
\sum_{r\leq x^{\theta}}\max_{(L_i(c),r)=1}\Big|\sum_{\substack{n\leq x\\n\equiv c\PMod r\\L_i(n)\in \mathbb{P}}}\omega_n-\frac{\varphi(a_i)}{\varphi(a_ir)}\sum_{\substack{n\leq x\\L_i(n)\in \mathbb{P}}}\omega_n\Big|\leq C_0 \frac{\sum_{n\leq x}\omega_n}{(\log x)^{101k^2}}.    
\end{align*}

\item (Brun--Titchmarsh) We have 
\begin{align*}
\max_{c\PMod r}\sum_{\substack{n\leq x\\n\equiv c\PMod r}}\omega_n\leq  \frac{C_0}{r}\sum_{n\leq x}\omega_n,    
\end{align*}
uniformly for $r\leq x^{\theta}$.
\end{enumerate}

Then, for $x\geq x_0(\theta, k, C_0)$, we have
\begin{align*}
\sum_{\substack{n\leq x\\ |\{L_1(n),\ldots, L_k(n)\}\cap \mathbb{P}|\geq C^{-1}\delta \log k\\ p\mid \prod_{i=1}^k L_i(n)\implies p>x^{\rho}}}\omega_n\gg_{k,\theta,\delta} \frac{\mathfrak{S}(\mathcal{L})}{(\log x)^k}\sum_{n\leq x}\omega_n,    
\end{align*}
where the singular series $\mathfrak{S}(\mathcal{L})$ is given by \eqref{eq27}.
\end{proposition}

\begin{proof}
This is  \cite[Theorem 6.2]{matomaki-shao} (with $\alpha=1$ there), which adds weights to the corresponding statement  in~\cite{MaynardII}.
\end{proof}

In the next sieve lemma for Chen primes, we need the notion of well-factorable weights.

\begin{definition} We say that a sequence $\lambda:[N]\to \mathbb{R}$ is \emph{well-factorable} of level $D\geq 1$, if for any $R,S\geq 1$ satisfying $D=RS$, we can write $\lambda=\lambda_1*\lambda_2$ for some sequences $|\lambda_1|,|\lambda_2|\leq 1$ supported on $[1,R]$ and $[1,S]$, respectively, with $*$ denoting Dirichlet convolution. 
\end{definition}

\begin{proposition}[Chen's sieve]  \label{prop: chen} Let $\varepsilon>0$ be a small enough absolute constant.  Let $(\omega_n)_{n\leq x}$ be any nonnegative sequence, let $\mathcal{L}=\{L_1,L_2\}$ with $L_i(n)=a_in+b_i$, $1\leq a_i\leq \log x$, $|b_i|\leq x$.  Suppose that $(\omega_n)$ satisfies the following hypotheses:

\begin{enumerate}[label=\upshape(\roman*)]
    \item (Bombieri--Vinogradov with well-factorable weights) We have
\begin{align*}
\Big|\sum_{\substack{r\leq x^{1/2-\varepsilon}\\(r,a_2(a_1b_2-a_2b_1))=1}}\lambda(r)\Big(\sum_{\substack{n\leq x\\L_2(n)\equiv 0\PMod r\\L_1(n)\in \mathbb{P}}}\omega_n-\frac{a_1}{\varphi(a_1 r)}\sum_{n\leq x}\frac{\omega_n}{\log L_1(n)}\Big)\Big|\ll \frac{\sum_{n\leq x}\omega_n}{(\log x)^{10}}    
\end{align*}
for any well-factorable sequence $\lambda$  of level $x^{1/2-\varepsilon}$, and also for $\lambda=1_{p\in [P,P')}*\lambda'$ with $\lambda'$ any well-factorable sequence of level $x^{1/2-\varepsilon}/P$ 
with $P' \in [P, 2P]$ and $P\in [x^{1/10},x^{1/3-\varepsilon}]$. 

\item (Bombieri--Vinogradov for almost primes with well-factorable weights) For $j\in \{1,2\}$ we have
\begin{align*}
\Big|\sum_{\substack{r\leq x^{1/2-\varepsilon}\\(r,a_1(a_1b_2-a_2b_1))=1}}\lambda(r)\Big(\sum_{\substack{n\leq x\\L_1(n)\equiv 0\PMod r\\L_2(n)\in B_j}}\omega_n-\frac{\varphi(a_2)}{\varphi(a_2r)}\sum_{\substack{n\leq x\\L_2(n)\in B_j}}\omega_n\Big)\Big|\ll \frac{\sum_{n\leq x}\omega_n}{(\log x)^{10}},    
\end{align*}
where $\lambda(r)$ is as above and 
\begin{align*}
B_1&=\{p_1p_2p_3:\,\, x^{1/10}\leq p_1\leq x^{1/3-\varepsilon},\,\, x^{1/3-\varepsilon}\leq p_2\leq (2x/p_1)^{1/2},\,\, p_3\geq x^{1/10}\},\\
B_2&=\{p_1p_2p_3:\,\, x^{1/3-\varepsilon}\leq p_1\leq p_2\leq (2x/p_1)^{1/2},\,\, p_3\geq x^{1/10}\}.
\end{align*}

\item (Upper bound on almost primes): For $j\in \{1,2\}$ we have
\begin{align*}
\sum_{\substack{n\leq x\\L_2(n)\in B_j}}\omega_n \leq (1+\varepsilon)\cdot \frac{|B_j\cap [1,L_2(x)]|}{\varphi(a_2)x}\sum_{n\leq x}\omega_n.  \end{align*}
\end{enumerate}
Then, for $x\geq x_0$, we have
\begin{align*}
\sum_{\substack{n\leq x\\ L_1(n)\in \mathbb{P}\\L_2(n)\in P_2\\p\mid L_2(n)\Longrightarrow p\geq x^{1/10}}}\omega_n\geq\delta_0 \frac{\FS(\mathcal{L})}{(\log x)^2}\sum_{n\leq x}\omega_n-O(x^{0.9}\max_{n}\omega_n),    
\end{align*}
for some absolute constant $\delta_0>0$, where the singular series $\mathfrak{S}(L)$ is given by \eqref{eq27}.
\end{proposition}

\begin{proof}
This is \cite[Theorem 6.4]{matomaki-shao} (which adds weights to Chen's sieve), with the slight modification that $|b_i|$ may be as large as $x$ (as opposed to $x^{o(1)}$). However, this restriction on $|b_i|$ was not used in the proof. Also, in \cite[Theorem 6.4]{matomaki-shao} $\lambda(r)$ was replaced with $\mu(r)^2\lambda(r)$, but since in the proof the sequence $\lambda(r)$ is always a sieve coefficient supported on squarefree numbers, this makes no difference.  
\end{proof}

For stating the weighted sieve for primes of the form $x^2+y^2+1$, we need a notion slightly different from admissibility, which  we call amenability, following \cite[Definition 3.1]{joni}.

\begin{definition} We say that a linear function $L(n)=Kn+b$ with $K\geq 1$ and $b\in \mathbb{Z}$ is \emph{amenable} if 
\begin{enumerate}[label=\upshape(\roman*)]
\item $6^3\mid K$;

\item $(b,K)=(b-1,s(K))=1$, where $s(n):=\prod_{p\mid n,p\equiv -1\PMod 4, p\neq 3}p$;

\item $b-1=2^j3^{2t}(4h+1)$ for some $h\in \mathbb{Z}$ with $3\nmid 4h+1$, and $j,t\geq 0$ with $2^{j+2}3^{2t+1}\mid K$. 
\end{enumerate}
\end{definition}

Here condition (ii) guarantees that there are no local obstructions to $L(n)$ being a prime of the form $x^2+y^2+1$. Conditions (i) and (iii) are introduced for technical reasons to do with sieves in \cite{joni}, but they are not very restrictive.

\begin{proposition}[Weighted sieve for primes of the form $x^2+y^2+1$] \label{prop: sumsoftwosquares}There exists some small $\varepsilon>0$ such that the following holds. Let $(\omega_n)_{n\leq x}$ be any nonnegative sequence, and let $L(n)=Kn+b$ be amenable with $1\leq K\leq \log x$.  Suppose that $(\omega_n)$ obeys the following hypotheses:

\begin{enumerate}[label=\upshape(\roman*)]
    \item For any sequence $(g(\ell))_{\ell}$ supported on $[1,x^{0.9}]$ and of the form $g=\alpha*\beta$ with $\alpha$ supported on $[x^{1/(3+\varepsilon)}, x^{1-1/(3+\varepsilon)}]$ and $|\alpha(n)|, |\beta(n)|\leq 1$, we have
\begin{align}\label{eq30}
\Big|\sum_{\substack{r\leq x^{1/2-\varepsilon}\\(r,K)=1}}\lambda_r^{+,\textnormal{LIN}}\sum_{\substack{\ell\leq x^{0.9}\\(\ell,K)=\delta\\(\ell,r)=1}}g(\ell)\Big(\sum_{\substack{n\leq x\\p\leq x\\L(n)=\ell p+1\\L(n)\equiv 0 \PMod r}}\omega_n-\frac{1}{\varphi(r)}\frac{K}{\varphi(\frac{K}{\delta})}\sum_{n\leq x}\frac{\omega_n}{\ell \log \frac{Kn}{\ell}}\Big)\Big|&\ll \frac{\sum_{n\leq x}\omega_n}{(\log x)^{100}},
\end{align}
where $\delta:=(b-1,K)$ and $\lambda_{r}^{+,\textnormal{LIN}}$ are the upper bound linear sieve coefficients of level $x^{1/2-\varepsilon}$ and sifting parameter $x^{\frac{1}{5}}.$ 

\item We have

\begin{align}\label{eq31}
\Big|\sum_{\substack{r\leq x^{3/7-\varepsilon}\\ (r,K)=1}}\lambda_r^{-,\textnormal{SEM}}\Big(\sum_{\substack{n\leq x\\L(n)\in \mathbb{P}\\L(n)\equiv 1 \PMod r}}\omega_n-\frac{1}{\varphi(r)}\frac{ K}{\varphi(K)}\sum_{n\leq x}\frac{\omega_n}{\log(Kn)}\Big)\Big|&\ll \frac{\sum_{n\leq x}\omega_n}{(\log x)^{100}},    
\end{align}
where $\lambda_{r}^{-,\textnormal{SEM}}$ are the lower bound semilinear sieve coefficients of level $x^{3/7-\varepsilon}$ and sifting parameter $x^{1/(3+\varepsilon)}$. 
\end{enumerate}

Then for some absolute constant $\delta_0>0$ we have
\begin{align}\label{eq26}
\sum_{\substack{n\leq x\\L(n)\in \mathbb{P}\\p\mid L(n)-1\implies p\not \equiv -1\PMod 4}}\omega_n\geq \delta_0\frac{\mathfrak{S}(L)}{(\log x)^{3/2}}\sum_{n\leq x}\omega_n-O(x^{1/2}),    
\end{align}
where the singular series $\mathfrak{S}(L)$ is given by
\begin{align}\label{singular_iwaniec}\begin{split}
\mathfrak{S}(L)&:=\prod_{\substack{p\equiv -1\PMod 4\\p\neq 3}}\left(1-\frac{|\{n\in \mathbb{Z}/p\mathbb{Z}:\,\, L(n)\equiv 0\,\, \textnormal{or}\,\, 1  \PMod p\}|}{p}\right)\left(1-\frac{2}{p}\right)^{-1}\\
 &\cdot  \prod_{p\not \equiv -1 \PMod 4}\left(1-\frac{|\{n\in \mathbb{Z}/p\mathbb{Z}:\,\, L(n)\equiv 0 \PMod p\}|}{p}\right)\left(1-\frac{1}{p}\right)^{-1}.  
 \end{split}
\end{align}
\end{proposition}

\begin{proof} This follows from \cite[Theorem 6.5]{joni}, taking $\rho_1=1/2-\varepsilon$, $\rho_2=3/7-\varepsilon$ and $\sigma=3+\varepsilon$ there and using the fact that hypothesis $H(\rho_1,\rho_2,\sigma)$ there holds with these parameters (the $n$ summation in \cite[Theorem 6.5]{joni} is over a dyadic interval, but this clearly makes no difference).
\end{proof}

\section{Bombieri--Vinogradov and Type I/II estimates for nilsequences}
\label{sec:BV}

In this section, we collect Bombieri--Vinogradov type estimates for nilsequences from~\cite{shao2020bombierivinogradov}  that we shall need. Theorems~\ref{equidist-1/3} and~\ref{equidist-1/2} below are slight generalizations of~\cite[Theorem 4.3]{shao2020bombierivinogradov} and~\cite[Theorem 4.4]{shao2020bombierivinogradov}, respectively.

\begin{theorem}\label{equidist-1/3}
Let an integer $s\geq 1$, a large real number $\Delta\geq 2$, and a small real number $\varepsilon\in (0,1/3)$ be given. Let $L(n) = an+b$ for some $1 \leq a \leq x^{\ee/2}$ and $|b| \leq x$ with $(a,b)=1$. There exists a constant $\kappa = \kappa(s,\Delta,\ee) > 0$, such that for any $x \geq 2$, $\eta>0$ and any nilsequence $\xi \in \Xi_s^0(\Delta, \eta^{-\kappa};\eta,x)$  we have
\[
\sum_{d\leq x^{1/3-\varepsilon}}\max_{(L(c),d)=1} \Big|\sum_{\substack{n \leq x \\ n \equiv c\PMod{d}}} \Lambda(L(n)) \xi(n)\Big|\ll\eta^{\kappa} a x (\log x)^2.
\]
\end{theorem}

\begin{theorem}\label{equidist-1/2}
Let integers $s\geq 1$, $c \neq 0$, a large real number $\Delta\geq 2$, and a small real number $\varepsilon\in (0,1/2)$ be given. Let $L(n) = an+b$ for some $1 \leq a \leq x^{\ee/2}$ and $|b| \leq x$ with $(a,b)=1$. There exists a constant $\kappa = \kappa(s,\Delta,\varepsilon) > 0$, such that for any well-factorable sequence $(\lambda_d)$ of level $x^{1/2-\varepsilon}$ with $x \geq 2$ and any nilsequence $\xi \in \Xi_s^0(\Delta, \eta^{-\kappa};\eta,x)$ with $\eta > 0$, we have 
\[
\Big|\sum_{\substack{d\leq x^{1/2-\varepsilon} \\ (d,c)=1}} \lambda_d \sum_{\substack{n \leq x \\ L(n) \equiv c\PMod{d}}} \Lambda(L(n)) \xi(n) \Big| \ll \eta^{\kappa} ax (\log x)^2.
\]
\end{theorem}

\begin{proof}[Proofs of Theorem~\ref{equidist-1/3} and~\ref{equidist-1/2}]
We deduce Theorem~\ref{equidist-1/3} from~\cite[Theorem 4.3]{shao2020bombierivinogradov}; the deduction of Theorem~\ref{equidist-1/2} from~\cite[Theorem 4.4]{shao2020bombierivinogradov} is completely similar.

Write $\xi(n) = F(g(n)\Gamma)$. One can find a polynomial sequence $g'$ such that $g'(L(n)) = g(n)$, for example, by examining the Taylor coefficients of $g$ in coordinates (see~\cite[Lemma 6.7]{GT-nilsequence}). 

We claim that $\{g'(n)\}_{n \leq ax}$ is totally $\eta^c$-equidistributed for some small constant $c = c(s,\Delta) > 0$. Suppose that this is not the case. Then by the quantitative Kronecker theorem for nilsequences (see~\cite[Theorem 2.9]{GT-nilsequence}),  there is a  nontrivial horizontal character $\chi$ with $\|\chi\| \ll \eta^{-O_{s,\Delta}(c)}$ such that
\[ \|\chi \circ g'\|_{C^{\infty}(ax)} \ll \eta^{-O_{s,\Delta}(c)}. \]
Since $g'$ is a polynomial sequence, we can write
\[ \chi \circ g'(n) = \alpha_0  + \alpha_1 n + \cdots + \alpha_s n^s. \]
Then there is a positive integer $q \ll_s 1$ such that the coefficients satisfy
\[ \|q\alpha_i\| \ll_s (ax)^{-i} \eta^{-O_{s,\Delta}(c)} \]
for each $1 \leq i \leq s$.
Now 
\[ \chi \circ g(n) = \chi \circ g'(an+b) = \sum_{i=0}^s \alpha_i (an+b)^i. \] If we write $\beta_j$ for the coefficient of $n^j$ in $\eta \circ g$, then one can establish that
\[ \|q \beta_j\| \ll_s x^{-j} \eta^{-O_{s,\Delta}(c)}  \]
for each $1 \leq j \leq s$. Hence
\[ \|q\chi \circ g\|_{C^{\infty}(x)} \ll_s \eta^{-O_{s,\Delta}(c)}. \]
It now follows (from~\cite[Lemma 3.6]{shao2020bombierivinogradov}) that $\{g(n)\}_{n \leq x}$ is not totally $\eta^{-O_{s,\Delta}(c)}$-equidistributed, which is a contradiction if $c$ is chosen small enough.

 Let $\xi'(n) = F(g'(n)\Gamma)$. Then $\xi' \in \Xi_s^0(\Delta, \eta^{-\kappa}; \eta^c, ax)$.
After a change of variables, we can write
\[ \sum_{\substack{n\leq x \\ n \equiv c\PMod{d}}} \Lambda(L(n))\xi(n) = \sum_{\substack{L(0) < n \leq L(x) \\ n \equiv L(c)\PMod{ad}}} \Lambda(n) \xi'(n).  \]
It follows that
\[ \sum_{d\leq x^{1/3-\varepsilon}}\max_{(L(c),d)=1} \Big|\sum_{\substack{n \leq x \\ n \equiv c\PMod{d}}} \Lambda(L(n)) \xi(n)\Big| \leq \sum_{d' \leq ax^{1/3-\ee}} \max_{(c',d')=1} \Big| \sum_{\substack{L(0) < n \leq L(x) \\ n \equiv c'\PMod{d'}}} \Lambda(n) \xi'(n) \Big|.
\]
By~\cite[Theorem 4.3]{shao2020bombierivinogradov}, the right-hand side above is $\ll \eta^{c\kappa} ax (\log x)^2$ for some constant $\kappa = \kappa(s,\Delta,\ee) > 0$. The conclusion follows.
\end{proof}

We will also need a few type I and type II estimates appearing in~\cite{shao2020bombierivinogradov}.

\begin{lemma}[Type I Bombieri--Vinogradov estimate]\label{typeI-BV}
Let $x \geq 2$ and $\varepsilon>0$. Let $1\leq M \leq x^{1/2}$ and $1\leq D\leq x^{1/2-\varepsilon}$.  Let $s \geq 1$, $\Delta \geq 2$, and $0 < \delta < 1/2$. Let $L(n) = an+b$ for some $1 \leq a \leq x^{\ee/2}$ and $|b| \leq x$ with $(a,b)=1$.
Let $\xi \in \Xi_s^0(\Delta, \delta^{-1}; \delta^C, x)$ for some sufficiently large constant  $C = C(s,\Delta,\varepsilon)$. Then
\[ \sum_{D\leq d\leq 2D} \max_{c\PMod d}  \sum_{\substack{M \leq m \leq 2M \\ (m,ad)=1}} \Big| \sum_{\substack{mn\leq L(x) \\mn \equiv c\PMod{d}\\mn\equiv b\PMod a}} \xi(L^{-1}(mn)) \Big| \ll \delta x. \]
\end{lemma}

\begin{proof}
The case $L(n)=n$ is \cite[Proposition 5.5]{shao2020bombierivinogradov}. We shall quickly reduce the general case to this case. 

From the arguments in the proof of Theorems~\ref{equidist-1/3} and~\ref{equidist-1/2}, there exists a nilsequence $\xi' \in \Xi_s^0(\Delta, \delta^{-1}; \delta^{C'}, ax)$ for some large constant $C' = C'(s,\Delta,\ee)$, such that $\xi'(n) = \xi(L^{-1}(n))$ if $n \equiv b\PMod{a}$. Then use the identity
\[ 1_{mn \equiv b\PMod{a}} = \frac{1}{\varphi(a)} \sum_{\chi\PMod{a}} \chi(m)\chi(n) \overline{\chi(b)} \]
to reduce matters to 
\[ \sum_{D\leq d\leq 2D}\max_{c\PMod d}\sum_{\substack{M\leq m\leq 2M\\(m,ad)=1}}\Big|   \sum_{\substack{mn\leq L(x) \\ mn \equiv c\PMod{d}}} \chi(mn)\xi'(mn)\Big|\ll \delta x  \]
for characters $\chi\pmod a$. Splitting $mn$ into residue classes $\PMod a$, it suffices to show for all $u$ coprime to $a$ that
\[ \sum_{D\leq d\leq 2D}\max_{c\PMod d}\sum_{\substack{M\leq m\leq 2M\\(m,ad)=1}}\Big|   \sum_{\substack{mn\leq L(x) \\ mn \equiv c\PMod{d}\\mn\equiv u\PMod a}} \xi'(mn)\Big|\ll \delta x  \]
Now, applying the Chinese remainder theorem to combine the congruences on $mn$, and making the change of variables $d'=[d,a]\leq 2aD$, the conclusion then follows from the case $L(n) = n$ that was already established.
\end{proof}

\begin{lemma}[Well-factorable type II Bombieri--Vinogradov estimate]\label{typeII-wellfactorable}
Let $\ee > 0$ be a small constant.
Let $x \geq 2$ and $M \in [x^{1/4}, x^{3/4}]$ be large and let $c \neq 0,k$ be fixed integers. Suppose that either
\begin{enumerate}[label=\upshape(\roman*)]
    \item $\lambda$ is well-factorable of level $x^{1/2-\varepsilon}$, or
    \item $\lambda= 1_{p \in [P,P')}*\lambda'$, where $\lambda'$ is well-factorable of level $x^{1/2-\varepsilon}/P$ and $2P\geq P'\geq P\in [x^{1/10},x^{1/3-\varepsilon}]$.
\end{enumerate} Let $s \geq 1$, $\Delta \geq 2$, $0 < \delta < 1/2$. Let $L(n) = an+b$ for some $1 \leq a \leq x^{\ee/2}$ and $|b| \leq x$ with $(a,b)=1$. Let $\xi \in \Xi_s^0(\Delta,\delta^{-1};\delta^C,x)$ for some sufficiently large constant $C = C(s,\Delta,\ee)$. Then
\begin{equation*}
\Big|\sum_{\substack{d \leq x^{1/2-\varepsilon} \\ (d,ac)=1}}  \lambda_d \sum_{\substack{L(x)\leq mn\leq L(2x) \\ M\leq m\leq 2M \\ mn \equiv c\PMod{d}\\ mn \equiv b\PMod{a}}} \alpha(m)\beta(n)\xi(L^{-1}(mn))\Big|  \ll \delta ax (\log x)^{O_k(1)},
\end{equation*}
uniformly for sequences $\{\alpha(n)\}$ and $\{\beta(n)\}$ satisfying $|\alpha(n)|, |\beta(n)| \leq d_k(n)$.
\end{lemma}

This is a consequence of the following somewhat more general statement.

\begin{lemma}\label{typeII-partially-wellfactorable}
Let $\ee > 0$ be a small constant.
Let $x \geq 2$ and $M \in [x^{1/4}, x^{3/4}]$ be large and let $c \neq 0,k$ be fixed integers. Let $R_1,R_2\geq 1$ be such that $R_1\leq x^{1-\varepsilon}/M$, $R_1R_2\leq x^{1/2-\varepsilon}$ and $R_1R_2^2\leq Mx^{-\varepsilon}$. Let $s \geq 1$, $\Delta \geq 2$, $0 < \delta < 1/2$. Let $L(n) = an+b$ for some $1 \leq a \leq x^{\ee/2}$ and $|b| \leq x$ with $(a,b)=1$. Let $\xi \in \Xi_s^0(\Delta,\delta^{-1};\delta^C,x)$ for some sufficiently large constant $C = C(s,\Delta,\ee)$. Then
\begin{equation*}
\sum_{\substack{R_1\leq r_1\leq 2R_1\\R_2\leq r_2\leq 2R_2 \\ (r_1r_2,ac)=1}}  \Big| \sum_{\substack{L(x)\leq mn\leq L(2x) \\ M\leq m\leq 2M \\ mn \equiv c\PMod{r_1r_2}\\ mn \equiv b\PMod{a}}} \alpha(m)\beta(n)\xi(L^{-1}(mn))\Big|  \ll \delta ax (\log x)^{O_k(1)},
\end{equation*}
uniformly for sequences $\{\alpha(n)\}$ and $\{\beta(n)\}$ satisfying $|\alpha(n)|, |\beta(n)| \leq d_k(n)$.
\end{lemma}

To see that Lemma \ref{typeII-wellfactorable} follows from Lemma \ref{typeII-partially-wellfactorable},  it suffices to show that the well-factorable sequence $\lambda$ can be decomposed into convolutions of the form $\gamma * \theta$, where the sequences $\gamma$ and $\theta$ are $1$-bounded sequences supported on $[1, 2R_1]$ and $[1, 2R_2]$, respectively, with $R_1 = x^{1-\ee}/M$ and $R_2 = Mx^{-1/2}$.
This is evidently true in case (i) of Lemma \ref{typeII-wellfactorable} since $\lambda$ is well-factorable.  In case (ii), since $P \leq R_1$, we can write $\lambda' = \lambda_1*\lambda_2$ for some sequences $\lambda_1,\lambda_2$ supported on $[1,R_1/P]$ and $[1,R_2]$, respectively. Then we can take $\gamma = 1_{p \in [P,P')}*\lambda_1$ and $\theta = \lambda_2$.

\begin{proof}[Proof of Lemma \ref{typeII-partially-wellfactorable}]
By switching the roles of $m$ and $n$ if necessary, we may assume that $M \in [x^{1/2}, x^{3/4}]$ The case $L(n) = n$ follows from \cite[Proposition 6.6]{shao2020bombierivinogradov}. The reduction of the general case to this case is very similar to the corresponding reduction in the proof of Lemma \ref{typeI-BV}.
\end{proof}



In the special case when $\lambda_d = 1_{d=1}$, Lemma~\ref{typeII-wellfactorable} implies that
\begin{equation}\label{eq:typeII-special}
\Big|\sum_{\substack{mn\leq L(x)\\M\leq m\leq 2M\\ mn \equiv b\PMod{a}}}\alpha(m)\beta(n)\xi(L^{-1}(mn))\Big|\ll_{k} \delta ax(\log x)^{O_{k}(1)}. 
\end{equation}
In the case $L(n) = n$, this is also the type II information required in Green and Tao's proof \cite[Section 3]{GT-mobius-nil} that the M\"obius function is orthogonal to nilsequences.


\section{Dealing with the equidistributed case}\label{sec:eq}


The goal of this section is to prove Propositions~\ref{prop:twin-prime-eq-nil} and~\ref{prop:chen-prime-eq-nil}. We shall apply the sieve lemmas in Section~\ref{sec:sieve} to reduce matters to certain Bombieri--Vinogradov type equidistribution results about primes weighted by nilsequences in arithmetic progressions, which follow from results in Section~\ref{sec:BV}.

\subsection{Proof of Proposition~\ref{prop:twin-prime-eq-nil}}

We may assume that $\varepsilon>0$ is fixed, since $x$ is large enough in terms of $\varepsilon$. In what follows, let $B$ be a large enough constant depending on $m,d,\Delta$. We may assume that $C$ is large enough in terms of $B$.
Recall Definition \ref{def:PsiDeltaKeta} and the notation from that definition, thus $\xi(n)=F(g(n)\Gamma)$, where $G/\Gamma$ is a nilmanifold equipped with a filtration of degree at most $d$, etc.
 Let $\mu=\int_{G/\Gamma}F$, so that the $\eta$-equidistribution of $\{g(n)\}_{n\leq x}$ implies
\begin{align}\label{equ3}
\Big|\sum_{n\leq x} (\xi(n)-\mu)\Big|\ll x/(\log x)^{B}. 
\end{align}
We may assume that $\mu\geq \varepsilon/2$, since otherwise \eqref{equ7} is trivial.

We will apply Maynard's sieve method in the form of Proposition \ref{prop: maynard}. We need to verify hypotheses (i)--(iv) there for the sequence $\omega_n=\xi(n)$ (with $\delta=1/2$ and $\theta=1/10$, say), and then the claim follows. 

Hypothesis (i) (with $\delta=1/2$ in its statement) asserts that
\begin{align*}
 \frac{\varphi(a_i)}{a_i} \sum_{\substack{n\leq x\\ L_i(n)\in \mathbb{P}}}\xi(n)\geq \frac{1}{2(\log x)}\sum_{n\leq x}\xi(n).   
\end{align*}
 Note that $\xi':= \xi-\mu$ is an equidistributed nilsequence lying in $\Xi_d^0(\Delta,K; \eta, x)$. By partial summation, we have
\begin{align}\label{equ2}
\Big|\sum_{\substack{n\leq x\\ L_i(n)\in \mathbb{P}}}\xi'(n)\Big|\ll\sup_{2\leq y\leq x}\frac{1}{\log L_i(y)}\Big|\sum_{n \leq y}\Lambda(L_i(n))\xi'(n)\Big|+O(x^{1/2}).    
\end{align}
we may apply (the $d=1$ case of) Theorem~\ref{equidist-1/3} to  bound the right-hand side of \eqref{equ2} by $\ll \eta^{\kappa}  x (\log x)^{A+10}$ for some constant $\kappa = \kappa(d,\Delta) > 0$, which can be made $\ll x(\log x)^{-B}$ by our assumption on $\eta$.  
Hypothesis (i) now follows from the prime number theorem and \eqref{equ3}.

We turn to hypothesis (ii), which (taking $\theta=1/10$ there) states that 
\begin{align}\label{equ4}
\sum_{r\leq x^{1/10}}\max_{c\PMod{r}}\Big|\sum_{\substack{n\leq x\\n\equiv c\PMod r}}\xi(n)-\frac{1}{r}\sum_{n\leq x}\xi(n)\Big|\ll x/(\log x)^{B}.    
\end{align}
We may clearly replace $\xi(n)$ by $\xi'(n)=\xi(n)-\mu$ here; the new nilsequence $\xi'$ lies in $\Xi_d^0(\Delta, K; \eta, x)$. Recalling \eqref{equ3}, our task is to show that
\begin{align*}
\sum_{r\leq x^{1/10}}\max_{c\PMod r}\Big|\sum_{\substack{n\leq x\\n\equiv c\PMod r}}\xi'(n)\Big|\ll x/(\log x)^{B}.    
\end{align*}
But this follows from Lemma \ref{typeI-BV} with $M = 1$ and $L(n)=n$.


Next we consider hypothesis (iii), which (with $\theta=1/10$) states that
\begin{align}\label{equ6}
\sum_{r\leq x^{1/10}}\max_{(L_i(c),r)=1}\Big|\sum_{\substack{n\leq x\\n\equiv c\PMod r\\ L_i(n) \in \mathbb{P}}}\xi(n)-\frac{\varphi(a_i)}{\varphi(a_ir)}\sum_{\substack{n\leq x\\ L_i(n)\in \mathbb{P}}}\xi(n)\Big|\ll x/(\log x)^{B}.
\end{align}
Applying the Bombieri--Vinogradov theorem, we may replace $\xi(n)$ by $\xi'(n)=\xi(n)-\mu$ on the left-hand side of \eqref{equ6}. By the argument we used to verify hypothesis (i), we have \begin{align}\label{equ16}
\Big|\sum_{\substack{n\leq x\\L_i(n)\in \mathbb{P}}}\xi'(n)\Big|\ll x/(\log x)^{B+1}.
\end{align}
Hence, by partial summation, \eqref{equ6} reduces to 
\begin{align}\label{equ24}
 \sum_{r\leq x^{1/10}}\max_{(L_i(c),r)=1}\Big|\sum_{\substack{n\leq y\\n\equiv c\PMod r}}\Lambda(L_i(n))\xi'(n)\Big|\ll x/(\log x)^{B}
\end{align}
for $y\in [x(\log x)^{-10B},x]$. This last claim follows from Theorem~\ref{equidist-1/3}.

Finally, hypothesis (iv) states that
\begin{align*}
\max_{c\PMod r}\sum_{\substack{n\leq x\\n\equiv c\PMod r}}\xi(n)\ll \frac{1}{r}\sum_{n\leq x}\xi(n).    
\end{align*}
However, this is trivial, since the left-hand side is $O(x/r)$ and the right-hand side is $\geq \varepsilon x/r$ by the consideration at the beginning of the proof and the fact that $\varepsilon>0$ is fixed.

This concludes the proof of Proposition \ref{prop:twin-prime-eq-nil}.

\subsection{Proof of Proposition~\ref{prop:chen-prime-eq-nil}}

We now turn to Chen primes.
 Let $\mu=\int_{G/\Gamma}F$. Similarly as in the proof of Proposition \ref{prop:twin-prime-eq-nil}, we may assume that $\mu\geq \varepsilon/2$, and we have \eqref{equ3}. 

We apply a weighted version of Chen's sieve from  Proposition \ref{prop: chen}. We see from it that the claim follows once we verify hypotheses (i)--(iii) there. 

Hypothesis (i) states that
\begin{align}\label{equ9}
\Big|\sum_{\substack{r\leq x^{1/2-\varepsilon'}\\(r,2a)=1}}\lambda(r)\Big(\sum_{\substack{n\leq x\\ L_2(n)\equiv 0\PMod r\\ L_1(n)\in \mathbb{P}}}\xi(n)-\frac{a}{\varphi(ar)}\sum_{\substack{n\leq x}}\frac{\xi(n)}{\log L_1(n)}\Big)\Big|\ll x/(\log x)^{10}    
\end{align}
for some small enough constant $\varepsilon'>0$, with $\lambda(r)$ either well-factorable of level $x^{1/2-\varepsilon'}$ or a convolution of the shape $1_{p\in [P,P')}*\lambda'$, with $\lambda'$ a well-factorable function of level $x^{1/2-\varepsilon'}/P $ and $2P \geq P'\geq P\in [x^{1/10},x^{1/3-\varepsilon'}]$. Note first that by the Bombieri--Vinogradov theorem we may replace $\xi$ with $\xi'=\xi-\mu$ on the left-hand side of \eqref{equ9} up to negligible error. Note also that 
$$\Big|\sum_{n\leq x} \frac{\xi'(n)}{\log L_1(n)} \Big|\ll x/(\log x)^B$$ 
by \eqref{equ3} and partial summation. Applying partial summation to replace $1_{\mathbb{P}}(L_1(n))$ with the von Mangoldt function, we are left with showing
\begin{align*}
\Big|\sum_{\substack{r\leq x^{1/2-\varepsilon'}\\(r,2a)=1}}\lambda(r)\sum_{\substack{n\leq y\\ L_2(n)\equiv 0\PMod r}}\Lambda(L_1(n))\xi'(n)\Big|\ll x/(\log x)^{100}    
\end{align*}
for all $y\in [x/(\log x)^{101},x]$. Since $L_2(n) = L_1(n) + 2$, the condition $L_2(n) \equiv 0\PMod{r}$ is equivalent to $L_1(n) \equiv -2\PMod{r}$. So this follows from Theorem~\ref{equidist-1/2}.


The statement of hypothesis (ii) is that, for $j\in\{1,2\}$, we have 
\begin{align}\label{equ10}
\Big|\sum_{\substack{r\leq x^{1/2-\varepsilon'}\\ (r,2a) = 1}}\lambda(r)\Big(\sum_{\substack{n\leq x\\ L_1(n)\equiv 0\PMod r\\ L_2(n)\in B_j}}\xi(n)-\frac{\varphi(a)}{\varphi(ar)}\sum_{\substack{n\leq x\\ L_2(n)\in B_j}}\xi(n)\Big)\Big|\ll x/(\log x)^{10}, 
\end{align}
where $\lambda(r)$ is as in hypothesis (i) and 
\begin{align*}
B_1&=\{p_1p_2p_3:\,\, x^{1/10} \leq  p_1\leq x^{1/3-\varepsilon'},\,\, x^{1/3-\varepsilon'}\leq p_2\leq (2x/p_1)^{1/2},\,\, p_3 \geq x^{1/10}\},\\
B_2&=\{p_1p_2p_3:\,\, x^{1/3-\varepsilon'} \leq p_1 \leq p_2\leq (2x/p_1)^{1/2},\,\, p_3\geq x^{1/10}\}.
\end{align*}
First note that $1_{n\in B_j}$ splits into a sum of $(\log x)^{10}$ type II convolutions $\alpha*\beta(n)$, where $|\alpha(n)|, |\beta(n)|\leq 1$ and $\alpha$ is supported on an interval $[M,2M]\subset [x^{1/3-\varepsilon'},x^{1/2}]$. Now, we decompose $\xi(n)=\xi'(n)+\mu$ and note that the contribution of the $\mu$ term to \eqref{equ10} is $\ll x/(\log x)^{B}$ by a type II Bombieri--Vinogradov estimate \cite[Theorem 17.4]{IwKo04} and the previous observation about $1_{n\in B_j}$ being a sum of type II convolutions. Now we shall prove that 
\begin{align}\label{equ12}
\Big|\sum_{\substack{n\leq x\\ L_2(n)\in B_j}}\xi'(n)\Big|\ll x/(\log x)^{B}.    
\end{align}
Again by the fact that $1_{L_2(n)\in B_j}$ is of type II, it suffices to prove after a change of variables that
\begin{align*}
\Big|\sum_{\substack{mn\leq L_2(x) \\ mn \equiv L_2(0)\PMod{a}}} \alpha(m)\beta(n)\xi'(L_2^{-1}(mn))\Big|\ll x/(\log x)^{2B}
\end{align*}
for any $|\alpha(n)|, |\beta(n)|\leq 1$, where $\alpha(n)$ is supported on an interval $[M,2M]\subset [x^{1/3-\varepsilon'},x^{1/2}]$. But this estimate follows from~\eqref{eq:typeII-special} as a special case of Lemma~\ref{typeII-wellfactorable}. Now we have reduced \eqref{equ10} to proving 
\begin{align}\label{equ11}
\Big|\sum_{\substack{r\leq x^{1/2-\varepsilon'}\\ (r,2a) = 1}}\lambda(r)\sum_{\substack{n\leq x\\ L_1(n)\equiv 0\PMod r\\ L_2(n)\in B_j}}\xi'(n)\Big|\ll x/(\log x)^{10}.
\end{align}
The condition $L_1(n) \equiv 0 \PMod{r}$ above is equivalent to $L_2(n) \equiv 2\PMod{r}$.
Once again recalling the type II nature of $1_{L_2(n)\in B_j}$ and using the well-factorable type II estimate of Lemma \ref{typeII-wellfactorable}, we obtain \eqref{equ11}.

We are left with hypothesis (iii), which states that
\begin{align}\label{equ8}
\sum_{\substack{n\leq x\\ L_2(n)\in B_j}}\xi(n)\leq (1+\varepsilon')  \frac{|B_j\cap [1,L_2(x)]|}{\varphi(a) x}\sum_{n\leq x}\xi(n)  
\end{align}
for $j\in \{1,2\}$ and for $\varepsilon'>0$ a small enough constant. 
This claim follows simply by decomposing $\xi(n)=\xi'(n)+\mu$ and using \eqref{equ3},~\eqref{equ12}, and the prime number theorem in arithmetic progressions.

All the hypotheses have now been verified, so the proposition follows.

\section{Proof of the main theorem}\label{sec: mainthm}

We now present the proof of our main theorem by combining the work in the previous sections.

\begin{proof}[Proof of Theorem \ref{th:linEqTwinPrimes}] 
We seek to apply Proposition \ref{prop: wtrick} to the functions $\theta=\theta_1$ (in which case 
$\mathcal{H}=\mathcal{H}_1=\{0,2\},r=2$) and $\theta=\theta_2$ (in which case $\mathcal{H}=\mathcal{H}_2$ is the tuple fixed when we defined $\theta_2$ and $r=m$).
To apply this theorem, we need to establish \eqref{eq:Wtricked}. 

Thus let $\Psi : \Z^d\rightarrow\Z^t$ be a system of finite complexity
whose linear coefficients are bounded in modulus by some constant $L$. 
Recall that by an easy linear algebraic argument, we may assume that $\Psi$ is in $s$-normal form
for some $s$.
Also suppose that $\mathcal{H}\subset [0,L]$.
Let $x\geq 1$ and $K\subset [-x,x]^d$ be a convex body such that  $\Psi(K)\subset [1,x]^t$ and $\Vol(K)\geq \eta x^d$ for some constant $\eta >0$.
Let $w\geq 1$ be chosen later (sufficiently large in terms of $d,t,L$) and let $W=\prod_{p\leq w}p$.
Let $(b_1,\ldots,b_t)\in B_\mathcal{H}^t$.
It suffices to prove that
\begin{equation}
\label{eq:hypoOfWtricked}
\sum_{\mathbf{n}\in\Z^d\cap K}\prod_{i\in [t]}\theta_{j,W,b_i}(\psi_i(\mathbf{n}))\gg_{d,t,L,\eta}
\Vol(K)
\end{equation}
for $j\in \{1,2\}$, where, recalling \eqref{eq: theta}, $\theta_{j,W,b}$ is defined as $\theta'_{W,b}$ for $\theta'=\theta_j$.
We will prove \eqref{eq:hypoOfWtricked} by appealing to Theorem \ref{th:transference}.
Let $M=M(d,t,L)$ be the constant produced by this theorem, and suppose that $x$ is large enough and 
$\alpha$ small enough as in this theorem.
Without loss of generality (upon using Bertrand's postulate, dilating $x$ by a factor of at most 8 and shrinking $\eta$ by a factor at most $8^d$),
we may assume that $x$ is prime and $K\subset [-x/4,x/4]^d$. 
By Proposition \ref{prop: pseudo}, there exists $c \in (0,1)$ and $C>0$ depending only on $d,t,L,M$
(and therefore ultimately on $(d,t,L)$ only) 
and an $(M,\alpha)$-pseudorandom measure $\nu_\mathbf{b} :\Z/x\Z\rightarrow\R_{\geq 0}$
such that $\theta_{j,W,b_i}(n)\leq C\nu(n)$ whenever $i\in [t]$ and $n\in [x^c,x]$.
Define then $\lambda_i : \Z/x\Z\rightarrow \R_{\geq 0}$ by $\lambda_i=\theta_{j,W,b_i}1_{[x^c,x]}/C$ where
as usual we identify $[x]$ and $\Z/x\Z$ in the natural way.
Therefore we have $\lambda_i\leq \nu$ on $\Z/x\Z$ by construction, so hypothesis (i) of Theorem \ref{th:transference} is satisfied.

We now turn to hypothesis (ii).
Let $\delta_1=\delta_0/(3C)$ where $\delta_0$ is the absolute constant of Proposition \ref{prop:chen-prime-nil} and $\delta_2$ be  the implied constant of Proposition \ref{prop:twin-prime-nil},
for our choice of $m$, divided by $3C$. 
Therefore $\delta_j$ depends at most on $d,t,L$ for each $j\in[t]$.
Let $Y_j,\varepsilon_j$ be the corresponding constants given by Theorem \ref{th:transference}, which are functions of $d,t,L,\delta_j$ for $j\in \{1,2\}$.
Fix $i\in [t]$.
For $j\in \{1,2\}$, denote by $f_j$ the function $\lambda_i$ constructed above from the function
$\theta=\theta_j$.
We intend to show that
\begin{equation}
\label{eq:sumnx}
\sum_{n\leq x}f_j(n)\xi(n)\geq \delta_j \sum_{n\leq x}\xi(n)
\end{equation} 
whenever $\xi : \Z\rightarrow [0,1]$ is a nilsequence of complexity at most $Y_j$ satisfying $\sum_{n\leq x}\xi(n)\geq \varepsilon_j x$.
Since $\sum_{n\leq x^c}\theta_{W,b}(n)\ll x^{c+o(1)}$, it suffices to show that
$\sum_{n\leq x}\theta_{W,b}(n)\xi(n)\geq 2C\delta_j\sum_{n\leq x}\xi(n)$.
But this follows from 
Propositions \ref{prop:chen-prime-nil} and \ref{prop:twin-prime-nil} assuming $x$ is large enough, and the diameter of $\mathcal{H}_2$ is smaller than $w$.

Therefore, the hypotheses of Theorem \ref{th:transference} are met, and applying this theorem yields equation \eqref{eq:hypoOfWtricked} and we are done.

%
%
%
Thus we obtain Theorem \ref{th:linEqTwinPrimes} with $C_i(\Psi)=\prod_p\beta_p(\Psi_{\mathcal{H}_i})$.
\end{proof}

\subsection{The case of primes of the form \texorpdfstring{$x^2+y^2+1$}{x2+y2+1}}

We now turn to the proof of Theorem~\ref{th: Linnik}. We shall be brief with the arguments at places, since for the most part they closely resemble those used to prove Theorem \ref{th:linEqTwinPrimes}.

Let $\theta_3(n)$ be the weighted indicator of primes of the form $x^2+y^2+1$ given by \eqref{equ21}. Also denote the set of sums of two squares by
\begin{align*}
S:=\{n\geq 1:\,\, n=x^2+y^2\,\, \textnormal{for some}\,\, x,y\in \mathbb{Z}\}.    
\end{align*}

We follow the proof strategy of Theorem \ref{th:linEqTwinPrimes}. Let $W:=6^3\prod_{3\leq p\leq w}p$. Define $\theta_{3,W,b}(n):=\left(W/\varphi(W)\right)^{3/2}\theta(Wn+b)$.

We first claim that Theorem~\ref{th: Linnik} follows if
\begin{align}\label{eq: theta3}
\sum_{\substack{\mathbf{n}\in 
\Z^d\\W\mathbf{n}+\mathbf{a}\in K}}\prod_{i=1}^t\theta_{3,W,\psi_i(\mathbf{a})}(\dot{\psi_i}(\mathbf{n}))\gg W^{-d}\Vol(K),   
\end{align}
for any convex body $K$ satisfying $\Psi(K)\subset [1,x]^d$, $\Vol(K)\gg x^d$ and for each $\mathbf{a}\in \mathcal{A}$, where $\mathcal{A}:=\{\mathbf{a}\in (\mathbb{Z}/W\mathbb{Z})^d:\, \forall\, i\in [t]\,\, Wn+\psi_i(\mathbf{a})\,\, \textnormal{amenable}\}$. The proof of this implication is essentially the same as for Proposition \ref{prop: wtrick}, i.e., we choose $K=K_{\mathbf{a}}$ as there and sum \eqref{eq: theta3} over all $\mathbf{a}\in \mathcal{A}$ and note that $|\mathcal{A}|\gg \prod_{p}\beta_p'W^d$ by the Chinese remainder theorem, where $$\beta_p':=\mathbb{E}_{\mathbf{a}\in (\mathbb{Z}/p\mathbb{Z})^d}\prod_{i\in [t]}(1-|A_p|/p)^{-1}1_{\psi_i(\mathbf{a})\not \in A_p\PMod p}$$ 
and $A_p=\{0,1\}$ for $p\equiv -1\PMod 4$ and $A_p=\{0\}$ otherwise. 

Thus, by applying Theorem \ref{th:transference}, it suffices to prove that the following hold for all fixed $w$ and  $1\leq b_i\leq W$ such that $Wn+b_i$ is amenable. 
\begin{enumerate}
    \item For any $M\geq 1,\alpha>0$, $t\geq 1$, there exist $0<c<1$ and an $(M,\alpha)$-pseudorandom measure $\nu_{\mathbf{b}}:\mathbb{Z}/x\mathbb{Z}\to \mathbb{R}_{\geq 0}$ such that $\theta_{W,b_i}(n)\ll_{M,t} \nu(n)$ whenever $i\in [t]$ and $n\in [x^{c},x]$.
    
    \item There exists an absolute constant $\delta_0>0$ such that,  for any $Y\geq 1$, $\varepsilon>0$ and $x\geq x_0(Y,\varepsilon)$ large enough, we have
\begin{equation*}
\sum_{n\leq x}\theta_{3,W,b_i}(n)\xi(n)\geq \delta_0 \sum_{n\leq x}\xi(n)
\end{equation*} 
whenever $\xi : \Z\rightarrow [0,1]$ is a nilsequence of complexity at most $Y$ satisfying $\sum_{n\leq x}\xi(n)\geq \varepsilon x$.
\end{enumerate}

\textbf{Proof of (1).} This follows from the work of 
Sun and Pan \cite{kAP-twoSquaresPlusOne} (Section 2 and in particular Proposition 2.1 there).

\textbf{Proof of (2).}
Let $Y$ and $\varepsilon$ be fixed in the statement of (2). For the proof of (2), it suffices to prove the following result, which is a direct analogue of Proposition \ref{prop:twin-prime-nil} in the case of bounded gap integers or of Proposition \ref{prop:chen-prime-nil}  in the case of Chen primes.

\begin{proposition}[Primes of the form $x^2+y^2+1$ with nilsequences]\label{prop:iwaniec-prime-nil}
Fix positive integers $d, \Delta$ and some  $\varepsilon>0$, $K \geq 2$. Also let $w\geq 1$ be sufficiently large in terms of $d,\Delta,\ee, K$ and $W=6^3\prod_{3\leq p\leq w}p$. The following statement holds for sufficiently large $x \geq x_0(d,\Delta,\varepsilon,K,w)$.

Let $\xi \in \Xi_d(\Delta,K)$ be a nilsequence taking values in $[0,1]$.
Then for some absolute constant $\delta_0>0$ and any $1\leq b\leq W$ such that $Wn+b$ is amenable, we have
\[ \sum_{\substack{n\leq x \\ Wn+b\in  \Prime\\ Wn+b-1\in S}} \xi(n) \geq \left(\frac{W}{\varphi(W)}\right)^{3/2} \frac{\delta_0}{(\log x)^{3/2}} \left( \sum_{n\leq x} \xi(n) - \varepsilon x \right). \]
\end{proposition}

By arguments similar to those in Section \ref{sec:twin-prime-bohr} (with Lemma \ref{lem:twim-prime-periodic} slightly adjusted to handle the local problem in our setting), we can reduce this to the equidistributed case.

\begin{proposition}[Primes of the form $x^2+y^2+1$ weighted by equidistributed nilsequences]\label{prop:iwaniec-eq-nil}
Fix positive integers $d, \Delta$ and some  $\varepsilon>0$, $A \geq 2$. There exists $C = C(d,\Delta) > 0$, such that the following statement holds for sufficiently large $x \geq x_0(d,\Delta,\varepsilon,A)$.

Let $K \geq 2$ and $\eta \in (0,1/2)$ be parameters satisfying the conditions
\[ \eta \leq K^{-C} (\log x)^{-CA}, \ \ K \leq (\log x)^C. \]
Let $\xi \in \Xi_d(\Delta, K; \eta, x)$ be a nilsequence taking values in $[0,1]$.  Let $L(n) = an+b$ be an amenable linear function, where $1 \leq a\leq \log x$, $|b|\leq x$. Then for some absolute constant $\delta_0>0$ we have
\begin{align}\label{eq: iwanieclower} \sum_{\substack{n\leq x \\ L(n)\in  \Prime\\ L(n)-1\in S}} \xi(n) \geq  \delta_0 \frac{\FS(L)}{(\log x)^{3/2}} \left( \sum_{n\leq x} \xi(n) - \varepsilon x \right), \end{align}
where the singular series is given by \eqref{singular_iwaniec}.
\end{proposition}

The remaining task is then to prove this proposition.

\begin{proof}[Proof of Proposition \ref{prop:iwaniec-eq-nil}]
We may assume that $\varepsilon>0$ is fixed, since $x$ is large enough in terms of $\varepsilon$. 

We apply Proposition~\ref{prop: sumsoftwosquares}. Thus, in order to obtain \eqref{eq: iwanieclower}, it suffices to verify  hypotheses (i)--(ii) of Proposition~\ref{prop: sumsoftwosquares} for $\omega_n=\xi(n)$ in order to obtain the claim. Since $\sum_{n\leq x}\omega_n\geq \varepsilon x\gg x$, it in fact suffices to verify versions of hypotheses (i)--(ii) where $(\sum_{n\leq x}\omega_n)/(\log x)^{100}$ is replaced with $x/(\log x)^{100}$ on the right-hand side of the inequalities \eqref{eq30}, \eqref{eq31}. 

Write $\xi(n)=\xi'(n)+\mu$, where $\mu=\int_{G/\Gamma}F$. Observing that hypotheses (i)--(ii) hold for constant sequences (in the case of (i) by a bilinear Bombieri--Vinogradov type estimate \cite[Theorem 17.4]{IwKo04} and in the case of (ii) by the classical Bombieri--Vinogradov theorem), it suffices to verify hypothesis (i)--(ii) (with $x/(\log x)^{100}$ on the right-hand side of \eqref{eq30}, \eqref{eq31}) for $\xi'(n)$, which belongs to $\Xi^0_d(\Delta, K; \eta, x)$ with $\eta\leq (\log x)^{-CA}$ for a large constant $C$.

\textbf{Verifying hypothesis (i).} Let $\omega_n=\xi'(n)$. First note that by partial summation and the fact that  $\xi'\in \Xi^0_d(\Delta, K; \eta, x)$, we have
\begin{align}\label{eq103}
\Big|\sum_{n\leq x}\frac{\xi'(n)}{\log(yn)}\Big|\ll \frac{x}{(\log x)^{300}},
\end{align}
say, uniformly for $x^{-0.99}\leq y\leq x$. Hence, recalling the definition of $g(\ell)$ in hypothesis (i), and denoting $u=(b-1,a)$, our task is to show that
\begin{align*}
\Big|\sum_{\substack{r\leq x^{1/2-\varepsilon}\\(r,a)=1}}\lambda_r^{+,\textnormal{LIN}}\sum_{\substack{x^{1/(3+\varepsilon)}\leq \ell_1\leq x^{1/(3+\varepsilon)}\\ \ell_2\leq x^{0.9-1/(3+\varepsilon)}\\(\ell_1\ell_2,a)=u\\(\ell_1\ell_2,r)=1}}\alpha(\ell_1)\beta(\ell_2)\Big(\sum_{\substack{n\leq x\\p\leq x\\L(n)=\ell_1\ell_2 p+1\\L(n)\equiv 0 \PMod r}}\xi'(n)\Big)\Big|&\ll \frac{x}{(\log x)^{100}},
\end{align*}
uniformly for $|\alpha(n)|,|\beta(n)|\leq 1$. Merging the variables $\ell_2$ and $p$ as $m=\ell_2 p$, it suffices to show that
\begin{align*}
&\Big|\sum_{\substack{r\leq x^{1/2-\varepsilon}\\(r,a)=1}}\lambda_r^{+,\textnormal{LIN}}\sum_{\ell_1\in I}\alpha(\ell_1)1_{(\ell_1,r)=1,(\ell_1,a)=u_1}b(m)1_{(m,r)=1,(m,a)=u_2}\Big(\sum_{\substack{n\leq x\\L(n)=\ell_1m+1\\L(n)\equiv 0 \PMod r}}\xi'(n)\Big)\Big|\\
&\ll \frac{x}{(\log x)^{100}},
\end{align*}
uniformly for $1\leq u_1,u_2\leq u$ and $|\alpha(n)|,|b(n)|\leq d_2(n)$, where for brevity we have denoted $I:=[x^{1/(3+\varepsilon)},x^{1-1/(3+\varepsilon)}]$. We make a linear change of variables in the inner sum over $n$ to reduce to
\begin{align*}
&\sum_{\substack{r\leq x^{1/2-\varepsilon}\\(r,a)=1}}|\lambda_r^{+,\textnormal{LIN}}|\Big|\sum_{\ell_1\in I}\alpha(\ell_1)1_{(\ell_1,r)=1,(\ell,a)=u_1}b(m)1_{(m,r)=1,(m,a)=u_2}\Big(\sum_{\substack{\ell_1m\leq L(x)\\\ell_1m\equiv 1 \PMod r\\\ell_1 m\equiv b-1\PMod a}}\xi'(\tfrac{\ell_1 m+1-b}{a})\Big)\Big|\\
&\ll \frac{x}{(\log x)^{101}}.
\end{align*}

We can replace the sequence $\lambda_{r}^{+, \textnormal{LIN}}$ above with a well-factorable sequence using \cite[Corollary 12.17]{Cribro}, which splits the linear sieve coefficients into a linear combination of boundedly many well-factorable sequences. Now the claimed estimate follows directly from the well-factorable type II estimate given by Lemma \ref{typeII-wellfactorable}.

\textbf{Verifying hypothesis (ii).} Let $\omega_n=\xi'(n)$. Again applying \eqref{eq103}, we reduce to
\begin{align*}
\Big|\sum_{\substack{r\leq x^{3/7-\varepsilon}\\ (r,a)=1}}\lambda_r^{-,\textnormal{SEM}}\Big(\sum_{\substack{n\leq x\\L(n)\in \mathbb{P}\\L(n)\equiv 1 \PMod r}}\xi'(n)\Big)\Big|&\ll \frac{x}{(\log x)^{100}}.    
\end{align*}
Making the change of variables $n'=L(n)$, this is equivalent to
\begin{align*}
\Big|\sum_{\substack{r\leq x^{3/7-\varepsilon}\\(r,a)=1}}\lambda_r^{-,\textnormal{SEM}}\Big(\sum_{\substack{n\leq L(x)\\n\in \mathbb{P}\\n\equiv 1\PMod{r}\\n\equiv b\PMod a}}\xi'(L^{-1}(n))\Big)\Big|\ll \frac{x}{(\log x)^{100}}.    
\end{align*} 
Applying partial summation, it suffices to show 
\begin{align*}
\Big|\sum_{\substack{r\leq x^{3/7-\varepsilon}\\(r,a)=1}}\lambda_r^{-,\textnormal{SEM}}\Big(\sum_{\substack{n\leq y\\n\equiv 1\PMod{r}\\n\equiv b\PMod a}}\Lambda(n)\xi'(L^{-1}(n))\Big)\Big|\ll \frac{x}{(\log x)^{100}},    
\end{align*} 
uniformly for $1\leq y\leq L(x)$. By Vaughan's identity, we can write the von Mangoldt function as a sum of $\ll (\log x)^{10}$ convolutions $a*b(n)$, where $|a(n)|,|b(n)|\leq (\log n)d_2(n)$ and $\textnormal{supp}(a)\subset [M,2M]$ and one of the following holds:
\begin{enumerate}
    \item $M\ll x^{1/3}$ and $b(n)\equiv 1$ or $b(n)\equiv \log n$ (type I case);
    \item $x^{1/2}\ll M\ll x^{2/3}$ (type II case).
\end{enumerate}
Thus, we reduce to proving that
\begin{align*}
\Big|\sum_{\substack{r\leq x^{3/7-\varepsilon}\\(r,a)=1}}\lambda_r^{-,\textnormal{SEM}}\Big(\sum_{\substack{mn\leq L(x)\\mn\equiv 1\PMod{r}\\mn\equiv b\PMod a}}a(m)b(n)\xi'(L^{-1}(mn))\Big)\Big|\ll \frac{x}{(\log x)^{110}}.    
\end{align*} 
In the type I case, we can in fact assume that $b(n)\equiv 1$ by applying partial summation. Now, to handle the type I sums,
we can apply the bound $|\lambda_{r}^{-,\textnormal{SEM}}|\leq 1$, followed by Cauchy--Schwarz to dispose of the $a(m)$ coefficients (this loses a factor of $(\log x)^{10}$, say, so for the resulting sum we need a bound of $x/(\log x)^{120}$). To the resulting sum we can then apply Lemma~\ref{typeI-BV} to obtain the desired conclusion. 

We then turn to the type II sums. If $M\leq x^{4/7}$, we can directly apply Lemma \ref{typeII-partially-wellfactorable} with $R_1=x^{3/7-\varepsilon}$, $R_2=1$, since then $R_1\leq x^{1-\varepsilon}/M$, $R_1R_2^2\leq Mx^{-\varepsilon}$. Hence, we may assume for now on that $x^{4/7}\leq M\ll x^{2/3}$.

We now apply a partial factorization of the lower bound semilinear sieve weights from
\cite[Lemma 9.2, formulas (10.3), (10.4)]{joni} (taking $\theta=\varepsilon/2$ and replacing $\varepsilon$ with $\varepsilon/10$ in those formulas). Since $x^{1/3-\varepsilon}\ll x^{1-\varepsilon}/M\ll x^{3/7-\varepsilon}$, we conclude that 
\begin{align}\label{eq104}
|\lambda_{r}^{-,\textnormal{SEM}}|\ll (\log x)^2 \max_{R_1,R_2}\sum_{\substack{r=r_1r_2\\r_1\in [R_1,2R_1]\\r_2\in [R_2,2R_2]}}1,    
\end{align}
where the maximum is over those $(R_1,R_2)\in \mathbb{R}_{\geq 1}^2$ satisfying
\begin{align}\label{eq105}
R_1\leq x^{1-\varepsilon}/M,\quad R_1R_2^2\leq Mx^{-\varepsilon},\quad R_1R_2\leq x^{3/7-\varepsilon/2}.    
\end{align}
Hence, applying \eqref{eq104}, the remaining task is to show that
\begin{align*}
\sum_{\substack{R_1\leq r_1\leq 2R_1\\R_2\leq r_2\leq 2R_2}}\Big|\sum_{\substack{mn\leq L(x)\\mn\equiv 1\PMod{r_1r_2}\\mn\equiv b\PMod a}}a(m)b(n)\xi'(L^{-1}(mn))\Big|\ll \frac{x}{(\log x)^{200}}  
\end{align*}
under the constraints \eqref{eq105}. Since the constraints on $R_1,R_2$ are precisely as in Lemma \ref{typeII-partially-wellfactorable}, we may appeal to that lemma to conclude. This completes the verification of hypothesis (i)--(ii), and hence the proof of Theorem \ref{th: Linnik}. 
\end{proof}

\end{document}